\documentclass[12pt,reqno]{amsart}
\usepackage{mathrsfs,epsfig,color,graphicx}
\def\I{{\rm I}}
\def\J{{\rm J}}
\def\K{{\rm K}}
\def\L{\mathscr{L}}
\def\B{\mathcal{B}}
\def\V{\mathcal{V}}
\def\DD{\,{\scriptstyle \Delta}\,}

\newtheorem{theorem}{Theorem}[section]   
\newtheorem{corollary}[theorem]{Corollary}     
\newtheorem{lemma}[theorem]{Lemma}         
\newtheorem{proposition}[theorem]{Proposition}  
\newenvironment{notation}{{\noindent\bf Notation.\ }\rm}{\bigskip}
\newenvironment{terminology}{{\noindent\bf Terminology.\ }\rm}{\bigskip}  
\newenvironment{remark}{{\noindent\bf Remark.\ }\rm}{\bigskip}  
\newenvironment{example}{\noindent\bf Example. \rm}{\bigskip}
\newenvironment{definition}{\noindent\bf Definition.\ \rm}{\bigskip}

\newcommand{\JS}[2]{\ensuremath{\mathrm{JS}_{#1}^{#2}}}
\newcommand{\HS}[2]{\ensuremath{\mathrm{HS}_{#1}^{#2}}}

\def\half{{\textstyle\frac12}}

\author{Philip Feinsilver}
\address{Department of Mathematics\\
Southern Illinois University \\
Carbondale, IL. 62901, U.S.A.}
\email{pfeinsil@math.siu.edu}

\title[$\lowercase{\mathrm{sl}(2)}$ in the Boolean lattice]{Representations of $\mathrm{sl}(2)$ in the Boolean lattice, and the Hamming and Johnson schemes}

\begin{document}

\begin{abstract}    
Starting with the zero-square ``zeon algebra", the regular representation gives rise to a Boolean lattice representation of $\mathrm{sl}(2)$.
We detail the $\mathrm{su}(2)$ content of the Boolean lattice, providing the irreducible representations carried by the algebra generated by
the subsets of an $n$-set. The group elements are found, exhibiting the ``special functions" in this context. The corresponding Leibniz rule and
group law are shown. Krawtchouk polynomials, the Hamming and the Johnson schemes appear naturally. Applications to the Boolean poset and the structure of Hadamard-Sylvester matrices are shown as well.
\end{abstract}

\maketitle

\thispagestyle{empty}
\begin{section}{Introduction}
The basic algebraic tool is the ``zeon algebra", a commutative
algebra which is generated by $n$ elements that square to zero.  A basis for the algebra is then naturally in one-to-one correspondence with the
subsets of the standard $n$-set, as for Clifford algebras. The study begins with the consideration of the regular representation of the basis
zeons acting on the algebra. 
These give rise to inclusion operators relating the subsets of an $n$-set. We thus find an action of the Lie algebra sl$(2)$ on the Boolean lattice.
The decomposition into irreducible representations in carried out in detail. We are interested as well in the associated group action, which we find explicitly. 
Various applications arise that are of combinatorial and algebraic interest. \bigskip

Here are some specifics. In \S2, we recall the Hamming and Johnson metrics. Then the zeon algebra is detailed. The appearance of the Lie algebra
sl$(2)$ leads to the next section, where the irreducible representations of su$(2)$, the relevant form of the Lie algebra in this context, are reviewed.
In \S4, the irreducible constituents of the representation on the Boolean algebra of subsets of an $n$-set, are found and their structure presented.
The next section discusses the group formed by exponentiation of the global representation and some examples/applications.
Then in \S6 we show how the Hamming scheme arises, where the Krawtchouk polynomials play an important r\^ole. A version of
the Bargmann transform comes into play connecting the matrices of the Johnson scheme with our algebra of sl(2) operators.

\begin{subsection}{References}
The book \cite{SST} from \S5.3 through Chapter 6, is directly related to this work. They work with the symmetric group rather than sl$(2)$.
Parts of the discussion of the Johnson scheme in \cite{BI} are close in spirit to the present work. Proctor \cite{PR} explicitly discusses the sl$(2)$ 
aspects of posets. Also see \cite{GO}. Godsil's notes on association schemes use local inclusion operators and are close in many points to the approach
taken here. Krawtchouk polynomials are well-known in the context of the Hamming scheme. They occur as orthogonal polynomials with respect
to the binomial distribution, see, e.g. \cite{Sz}. Our work has many connections with the book of Louck, \cite{L}, to which our presentation may
serve as a useful complement.
\end{subsection}
\end{section}

\section{Boolean lattice}
Fix a positive integer $n$ and consider the Boolean lattice, $\B$, of subsets of the standard $n$-set, $\{1,2,\ldots,n\}$.  In contexts where
$n$ may vary, we write explicitly $\B=\B(n)$. \bigskip

\begin{notation} 
 Roman capitals $\I$, $\J$, etc., denote subsets of $\{1,\ldots,n\}$, and will be used both as unordered multi-indices,  e.g., $\I=\{i_1,i_2,\ldots,i_m\}$, $m\le n$, or
as ordered multi-indices $\I=(i_1,i_2,\ldots,i_m)$ where $i_1~<~i_2~<~\cdots~<~i_m$. 
The complement of $\I$ in $\{1,\ldots,n\}$ is indicated by a prime: $\I'$.\\
We will use math italic $I$  for the identity matrix.\\
 Our usage is that ${\rm A}\subset {\rm B}$ does not require proper inclusion.\\
For binomial coefficients, we require for $\binom{a}{b}$ to be nonzero that $b$ be a nonnegative integer.\\
Partitioning by cardinality, a \textsl{layer} refers to all subsets with the same number of elements. 
For any subset $\I$, $|\I|$ denotes the number of elements in $\I$. We denote the $\ell^{\rm th}$ layer by
$$\B_{\ell}=\{ \I\in\B\colon |\I|=\ell\} \ .$$
\end{notation}

\begin{subsection}{Hamming and Johnson metrics}
Recall the Hamming metric on $\B$  given by
$$\text{dist}_{H}(\I,\J)=|\I\DD\J|$$
where $\DD$ denotes symmetric difference. The Johnson metric on $\B_\ell$ is
$$\text{dist}_{J}(\I,\J)=\ell-|\I\cap\J|={\textstyle\frac12}\,|\I\DD\J|$$
half the Hamming distance. Each metric defines a family of relations matrices. For the Hamming metric, define the $2^n\times 2^n$ matrix
$\HS{j}{n}$, the $j^{\text{th}}$ relation matrix for the Hamming scheme on $n$ points, with entries
\begin{equation}\label{eq:hamming}
(\HS{j}{n})_{\I\J}=\begin{cases} 1, & \text{ if } |\I\DD\J|=j  \\
0,& \text{ otherwise }
\end{cases}
\end{equation}
the indicator of the relation $\text{dist}_H(\I,\J)=j$. Similarly for each layer $\B_\ell$, define the $\binom{n}{\ell}\times\binom{n}{\ell}$ matrix
$$(\JS{j}{n\ell})_{\I\J}=\begin{cases} 1, & \text{ if } |\I\cap\J|=\ell-j  \\
0, & \text{ otherwise }
\end{cases}$$
the indicator of the relation $\text{dist}_{J}(\I,\J)=j$.
\end{subsection}

\begin{subsection}{Zeon algebra} \label{subsec:zeon}
Consider a vector space $\V_1=\mathbb{C}^n$ with basis $\{e_1,e_2,\ldots,e_n\}$. We define the \textsl{zeon algebra} to 
be the algebra over $\mathbb{C}$ generated by the elements $e_i$ with the relations
\begin{align*}
e_ie_j&=e_je_i\,,&&\forall\, i,j\\
e_i^2&=0\,,&&\forall\, i
\end{align*}
in summary: the $e_i$ commute and square to zero. The name zeon is a variation on ``boson" and ``fermion" with the ``ze" short for ``zero". \bigskip

A basis for the algebra is given by the products indexed by subsets
$$e_\I=e_{\{i_1,i_2,\ldots,i_\ell\}}=e_{i_1}e_{i_2}\cdots e_{i_\ell}$$
We denote by $\V_\ell$ the subspace generated by $e_\I$, $\I\in\B_\ell$.  \bigskip

Adjoin $\V_0=\text{span}\{e_\emptyset\}$, the one-dimensional space generated by $e_\emptyset$. Then the zeon algebra is identified with the vector space
$$\V=\bigoplus_{\ell=0}^n \V_\ell$$
where in the algebra, $e_\emptyset=1$, the multiplicative identity. \bigskip

\begin{definition}
The basis $\{e_\I\}$, where $\I$ runs through $\B(n)$, is called the \textit{natural basis}.
\end{definition}

\begin{subsubsection}{Orthogonality}
We take the standard inner product whereby the $e_i$ are an orthonormal system and extend it to $\B$ so that
$$\langle e_\I,e_\J\rangle =\delta_{\I\J}\ .$$
Thus the $e_\I$ are an orthonormal system.
\end{subsubsection} \bigskip

\begin{subsubsection}{The $\ \hat{}$-representation}

Now consider the linear operator $\hat e_i$ of multiplication by $e_i$. 
$$ \hat e_i\, e_\I=\begin{cases} e_{\{i\}\cup \I}, & \text{if } i\notin\I\\
0, &\text{otherwise}
\end{cases}
$$

Consider the dual basis $\{\delta_i\}$. The action of $\delta_i$ is given by the linear operator $\hat\delta_i$ defined by
$$ \hat \delta_i\, e_\I=\begin{cases} e_{\I\,\setminus\,\{i\}},  & \text{if } i\in\I\\
0, &\text{otherwise}
\end{cases}
$$

We check that
$$\langle \hat e_i e_\I,e_\J\rangle = \langle e_\I,\hat \delta_i e_\J\rangle$$
indeed identifying $\hat \delta_i=\hat e_i^*$, the adjoint of $\hat e_i$.
Together the $\hat e_i$'s and $\hat \delta_i$'s generate a *-algebra. \bigskip

The operators $\hat e_i$, $\hat\delta_i$ satisfy the anticommutation relations
\begin{equation}\label{eq:fermi}
\hat e_i\hat\delta_i+\hat\delta_i\hat e_i=I
\end{equation}

as is easily seen by considering the action on $e_\I$ according to whether $i\in\I$ or not. On the other hand this shows that the commutator, call it
$\hat h_i$, satisfies
\begin{equation}\label{eq:bose}
\hat h_i=\hat\delta_i\hat e_i-\hat e_i\hat\delta_i=I-2\hat e_i\hat\delta_i
\end{equation}

\begin{proposition} \label{prop:LA}
The operators $\{\hat e_i,\hat\delta_i,\hat h_i\}$ satisfy the commutation relations
$$[\hat\delta_i,\hat e_i]=\hat h_i\,,\qquad [\hat e_i,\hat h_i]=2\hat e_i\,,\qquad [\hat h_i,\hat\delta_i]=2\hat\delta_i$$
\end{proposition}
\begin{proof} Start with
\begin{equation}\label{eq:triples}
 \hat e_i\hat\delta_i\hat e_i=\hat e_i \,,\qquad \hat\delta_i\hat e_i\hat\delta_i=\hat\delta_i
\end{equation}
which are seen by acting on $e_\I$ and considering whether $i\in\I$.
Now, the result follows directly upon writing out the commutators. 
\end{proof}
From equation \eqref{eq:triples} we have
\begin{corollary} 
The operators $\hat e_i\hat\delta_i$ and $\hat\delta_i\hat e_i$ are idempotents satisfying 
$$\hat e_i\hat\delta_i+\hat\delta_i\hat e_i=I\ \mathrm{and}\ (\hat e_i\hat\delta_i)(\hat\delta_i\hat e_i)=(\hat\delta_i\hat e_i)(\hat e_i\hat\delta_i)=0$$
i.e., they are complementary projections.
\end{corollary}
Observe that the projections $\hat e_i\hat \delta_i$ along with the identity generate a commutative algebra. \bigskip
\end{subsubsection}
\begin{subsubsection}{Incidence matrices for inclusion}
Define $T=\sum \hat e_i$, $T^*=\sum\hat\delta_i$. We will identify $T$ and $T^*$ with their matrices in the natural basis $\{e_\I\}$.
Note that (the matrix of) $T^*$ is the transpose of (the matrix of) $T$. \bigskip

Thus $T$ is the $n\times n$ matrix with entries
$$ T_{\I\J}=\langle e_\I,Te_\J\rangle=\begin{cases} 1, & \text{ if } \I\supset \J \text{ and } |\I\setminus\J|=1 \\
0,& \text{ otherwise }
\end{cases}
$$
with $T^*$ its transpose.  Note that $T+T^*=\HS{1}{n}$, where sets $\I$ and $\J$ are unit distance apart if they differ by adjoining or removing
a single element. \bigskip

Observe that 
$$T^2=\sum_{i\ne j}e_ie_j=2\sum_{|\I|=2} e_\I$$
and in general, each set appears $k!$ times in $T^k$, so that
\begin{equation}\label{eq:elemsymm}
T^k/k!=\sum_{|\I|=k} e_\I
\end{equation}
the $k^{\text{th}}$ elementary symmetric function in the variables $e_i$. 
Now consider the exponential series for $e^{zT}$. Start with the zero-square property to note
$$e^{ze_i}=1+ze_i$$
Thus
\begin{align*}
e^{zT}&=\sum_{k\ge0}\frac{z^k}{k!}T^k=\prod_i e^{ze_i}\\
&=\prod_i(1+ze_i)=\prod_k z^k\sum_{|\I|=k} e_\I
\end{align*}
in agreement with \eqref{eq:elemsymm}. Using \eqref{eq:elemsymm} we can derive 
$$(T^k/k!) e_\J=\sum_{|\K|=k}e_\K e_{\J}=\sum_{\substack{ \I=\K\cup\J\\|\K|=k}}e_{\I}$$
so that $\I$ appears in the sum if and only if $\K\cap\J=\emptyset$ and $\I\supset\J$ differing by exactly $k$ elements.
Thus the matrix elements
\begin{equation}\label{eq:inclusion}
(T^k/k!)_{\I\J}=\langle e_\I,(T^k/k!) e_\J\rangle=\begin{cases} 1, & \text{ if } \I\supset \J \text{ and } |\I\setminus\J|=k \\
0,& \text{ otherwise }
\end{cases}
\end{equation}
the incidence matrix for inclusion where the sets differ by $k$ elements. 
\end{subsubsection}

\begin{subsubsection}{Matrix realizations of the zeon algebra}
The matrices for the linear maps $\hat e_i$ give a realization of the zeon algebra. Another approach is to use explicit Kronecker products of
matrices. Start with $R=\begin{pmatrix} 0&0\\1&0\end{pmatrix}$. Then we have the isomorphism induced by the correspondences
\begin{align*}
e_i\,&\leftrightarrow\, I\otimes \cdots \otimes I\otimes R\otimes I\otimes \cdots \otimes I\\
\delta_i\,&\leftrightarrow\, I\otimes \cdots \otimes I\otimes R^*\otimes I\otimes \cdots \otimes I\\
\end{align*}
$R$ in the $i^{\rm th}$ spot, the * denoting transpose. With $R^2=0$, the verification is immediate. The corresponding matrices can be calculated
iteratively by associating to the left, resulting in $2^n\times 2^n$ matrices. Using binary labelling of rows and columns of the tensor products,
it is readily seen that this realization is permutation equivalent to the $\hat{}$-representation, identifying each $n$-digit binary label with the corresponding
subset of $\{1,\ldots,n\}$. Note the correspondence of the $e_i$ alone or of the $\delta_i$ alone
gives an isomorphism with the zeon algebra. Including the $e$'s and $\delta$'s together gives
an isomorphism of Lie algebras, correspondence of commutators, and of Jordan algebras, the anticommutators corresponding.
\end{subsubsection}

\begin{subsubsection}{Complementation}
The complementation symmetry of $\B$ yields an involution on $\V$. Denote by $\I'$ the complement of $\I$ in $\{1,2,\ldots,n\}$. Then the involution
on $\V$ is defined by
$$\text{for }\phi=\sum_\I c_\I e_\I \quad\text{define}\quad \phi'=\sum_\I c_\I e_{\I'}\ .$$
Compare the action of the operators $\hat e_i$ and $\hat \delta_i$:
$$\begin{aligned}
\hat e_i\, e_{\I'}=\begin{cases} e_{\I'\cup \{i\}}, & \text{if } i\in\I\\
0, &\text{otherwise}
\end{cases}
\end{aligned} \ ,\qquad
\begin{aligned}
\hat \delta_i\, e_{\I}=\begin{cases} e_{\I\,\setminus\,\{i\}},  & \text{if } i\in\I\\
0, &\text{otherwise}
\end{cases}
\end{aligned}
$$
Noting
$$(\I\setminus \{i\})'=\I'\cup\{i\}$$
we deduce the relations
\begin{equation}\label{eq:comp}
 T\phi'=(T^*\phi)' \qquad\text{and}\qquad T^*\phi'=(T\phi)'
\end{equation}
for $\phi\in\V$. \bigskip

\begin{paragraph}{\textsl{Complementation reverses dictionary order}}
Perhaps this is obvious. In any case, to gain some familiarity with our context, let us show it.
First, an observation about the Johnson metric.
\begin{proposition}
If $\I,\J\in \B_\ell$, then $\text{dist}_{J}(\I',\J')=\text{dist}_{J}(\I,\J)$. That is, complementation is an isometry for the Johnson metric.
\end{proposition}
\begin{proof}This is readily seen by a useful diagram (Figure \ref{fig:Johnson}). If $\text{dist}_{J}(\I,\J)=k$, we note that
$|\I\setminus\J|=|\J\setminus\I|=k$. Noting the definition, $\I\setminus\J=\I\cap\J'$, we see directly that
$\I\setminus\J$ and $\J\setminus\I$ are switched upon complementation. Thus the Johnson metric is preserved.
\begin{figure}[ht!]
\begin{center}
\scalebox{.5}{\input{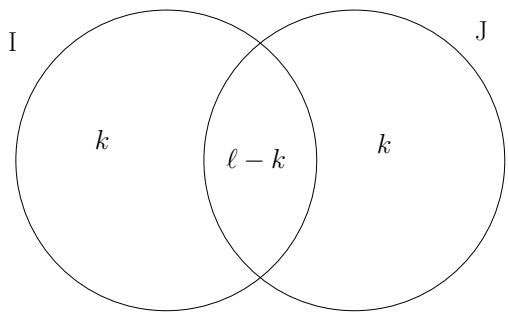}}
\caption{Johnson distance $\text{dist}_{J}(\I,\J)=k$.}\label{fig:Johnson}
\end{center}
\end{figure}
\end{proof}
For comparison of two sets of size $\ell$, we write the elements as a string. Take $n=9$, $\I=\{1,3,6,7,8,9\}$,
$\J=\{1,3,4,5,7,8\}$. We identify these with strings for comparison
$$\begin{matrix}& 134578\\&136789\end{matrix}$$
Identify the initial matching segment, the common prefix, $p$, the first place they differ $a$ vs.\,$A$, and the rest, suffixes $s$, $S$ :
$$\begin{matrix}& pas\\&pAS\end{matrix}$$
In our example, $p=13$, $a=4$, $A=6$, $\I\cap\J=\{1,3,7,8\}$.
Assume $a<A$. We call $a$ the {\it comparison value}, cv. Certainly, everything appearing in $p$ is
in $\I\cap\J$. By definition, cv$=a$ is not in $\I\cap\J$, otherwise it would extend $p$. But $A$ could be in $\I\cap\J$. In any case,
cv$=\min (\I\DD\J)$. If cv$\in\I\setminus\J$, it moves to $\J\setminus\I$ under complementation, and vice versa,
so the order of the strings is reversed.
\end{paragraph}
\end{subsubsection}

\end{subsection}

\section{Representations of {\rm su}$(2)$}
Starting with the Pauli matrices
$$\sigma_x=\begin{pmatrix}0&1\\1&0\\ \end{pmatrix}\,,\qquad \sigma_y=\begin{pmatrix}0&-i\\i&0\\ \end{pmatrix}\,,\qquad
\sigma_z=\begin{pmatrix}1&0\\ 0&-1\\ \end{pmatrix} $$
form the complex linear combinations
$$R=\textstyle{\frac12}(\sigma_x-i\sigma_y)=\begin{pmatrix}0&0\\1&0\\ \end{pmatrix}\,,\qquad
L=\textstyle{\frac12}(\sigma_x+i\sigma_y)=\begin{pmatrix}0&1\\0&0\\ \end{pmatrix}$$
and let $H=\sigma_z$. Then $R$ is a \textsl{raising operator}, $L$ a \textsl{lowering operator} for representations of $\text{su}(2)$ as will be seen shortly.\\

The defining commutation relations are thus:
$$[L,R]=H\,,\qquad [R,H]=2R \,,\qquad [H,L]=2L$$
Working in the enveloping algebra, adjoining the identity operator $I$, it is readily checked that the Casimir operator
$$C=4RL+(H+I)^2$$
commutes with each of $R$, $L$, and $H$.  \bigskip

\begin{terminology}
Three operators in correspondence with $R$, $L$, and $H$, satisfying the above commutation relations are called a \textsl{standard triple}.
\end{terminology}

The irreducible representations are constructed as follows. Choose an integer $N\ge0$. Let $\phi_0$ be a \textsl{vacuum vector}, i.e.,
determined by the conditions
$$ L\phi_0=0\,,\qquad C\phi_0=(N+1)^2\phi_0$$
The condition on $C$ is equivalent to
$$H\phi_0=N\phi_0$$
A basis for the representation is the sequence of vectors
$$\phi_j=R^j\phi_0$$
where $0\le j\le N$. The Casimir operator satisfies $C\phi_j=(N+1)^2\phi_j$, for $0\le j\le N$. Observe that the representation has
dimension $N+1$. Inductively, in the enveloping algebra, 
$$HR=R(H-2I)$$
so that
$$H\phi_j=R^j(H-2jI)\phi_0=(N-2j)\phi_j$$
And one checks that for $j\ge1$,
$$LR^j-R^jL=\sum_{k=0}^{j-1}R^{j-k-1}HR^k$$
So
$$LR^j\phi_0=\sum_{k=0}^{j-1}(N-2k)\phi_{j-1}$$
Thus, the action on the basis is given by
\begin{equation}\label{eq:su2}
R\phi_j=\phi_{j+1}\,,\qquad L\phi_j=j(N+1-j)\phi_{j-1}\,,\qquad H\phi_j=(N-2j)\phi_j
\end{equation}
Conversely, defining $R$, $L$, and $H$ by these formulas, the required commutation relations
are seen to hold.  An alternative basis is often useful, scaling by $j!$. I.e., define
$$\tilde\phi_j=(R^j/j!)\phi_0=\phi_j/j!$$
The action on this basis is
\begin{equation}\label{eq:tildeaction}
R\tilde\phi_j=(j+1)\tilde\phi_{j+1},\ L\tilde\phi_j=(N+1-j)\tilde\phi_{j-1},\ H\tilde\phi_j=(N-2j)\tilde\phi_j
\end{equation}
Note that both $RL$ and $LR$ are diagonal in either basis with
$$RL\phi_j=j(N+1-j)\phi_j\,,\qquad LR\phi_j=(j+1)(N-j)\phi_j$$
In particular, $RL$, $LR$ and $H$ all commute. \bigskip

\begin{notation}     
We use boldface $\bf d$ to denote the irreducible representation of dimension $d$, so that $N=d-1$.
The corresponding basis elements will be denoted $\phi_j(N)$. 
\end{notation}

\begin{remark} In the physics literature for angular momentum, the letter $j$ is used to describe the representation where
$j=N/2$ in our notation. Typically states are labelled by kets $|jm\rangle$ with $-j\le m\le j$. In the present context, it is more
natural to use the variant notation we have indicated. Further on, we will use kets to denote the states comprising the su(2) content of $\B$.
\end{remark}

\begin{terminology}
We call $N$ the \textsl{principal number} of an irreducible representation. By a \textsl{vacuum state}, denoted $\Omega$, we mean a vector satisfying
$L\Omega=0$, $H\Omega=N\,\Omega$. It plays the r\^ole of $\phi_0$ in the above discussion. $R$ acts nilpotently on $\Omega$, with
$R^{N+1}\Omega=0$.
\end{terminology}

\begin{subsection}{Orthogonality}
With $R$ and $L$ mutual adjoints, the inner product, taking $k\ge j$,
\begin{align*}
\langle\phi_j,\phi_k\rangle&=\langle\phi_j,R^k\phi_0\rangle=\langle L^k\phi_j,\phi_0\rangle\\
&=j(j-1)\cdots(j-k+1)\,(N+1-j)\cdots(N+k-j)\langle\phi_0,\phi_0\rangle
\end{align*}
So we have an orthogonal system with
$$\|\phi_j\|^2=j!\,N^{(j)}\|\phi_0\|^2$$
the superscript denoting a falling factorial. Note that for the basis $\{\tilde\phi_j\}$ we have
\begin{equation}\label{eq:ortho1}
\|\tilde\phi_j\|^2=\binom{N}{j}\|\phi_0\|^2
\end{equation}

a convenient scaling. We will see later that it is convenient to work with vacuum states $\phi_0$ which are not normalized.
\end{subsection}

\begin{subsection}{Tensoring}
In general, tensoring $\bf d_1$ with $\bf d_2$ yields a direct sum of irreducible representations, the \textsl{Clebsch-Gordan series}.
For the Boolean case, we only will do this for tensor products with $\bf 2$ for which the formula
$${\bf d}\otimes {\bf 2}=({\bf d-1})\oplus({\bf d+1})$$
holds. We will show explicitly how this decomposition works in general and we will use it in the next section to generate the su(2) content of $\B$. \bigskip

First we describe the construction in general. Suppose we have a reducible representation of the Lie algebra, with standard triple $\{R,L,H\}$.
Then to find the irreducible subrepresentations, find a linearly independent set of vacuum states, $\Omega_i$ satisfying $L\Omega_i=0$, $H\Omega_i=N_i\Omega_i$. 
Fix a vacuum state $\Omega_i$. Applying $R$ generates the basis for that subrepresentation, with $N_i$ being the smallest index $j$ 
such that $R^{j+1}\Omega_i=0$. Iterating through the vacuum states thus generates the representation as a sum of irreducible representations. \bigskip

For ${\bf d}\otimes {\bf 2}$, with $d=N+1$, let $r_N$, $l_N$, $h_N$ denote a standard triple for the representation with
principal number $N$, $r_1$, $l_1$, $h_1$ correspondingly for the $\bf 2$. Then the combined system has standard triple
$$R=r_N+r_1 \,,\qquad L=l_N+l_1 \,,\qquad H=h_N+h_1$$
with all $1$-indexed operators commuting with all $N$-indexed ones. Explicitly written out using tensor notation, we have
$r_N\leftrightarrow r_N\otimes I$, $r_1\leftrightarrow I\otimes r_1$, etc. Let $\{\phi_j\}$ denote the basis for the $\bf d$, and $\{\psi_0,\psi_1\}$ for
the $\bf 2$. \bigskip

For the first subrepresentation, take $\Omega=\phi_0\otimes\psi_0$. For simplicity of notation, we will drop explicit tensor signs, thus
$\Omega=\phi_0\psi_0$. With $r_1^2=0$, we have a representation with principal number $N+1$ with basis $f_j(N+1)$. We calculate
$$f_j(N+1)=R^j\Omega=(r_N^j+jr_N^{j-1}r_1)\Omega=\phi_j\psi_0+j\phi_{j-1}\psi_1$$
We check, with $h_N\phi_j=(N-2j)\phi_j$, $h_1\psi_0=\psi_0$, $h_1\psi_1=-\psi_1$,
$$Hf_j=(h_N+h_1)f_j=(N-2j+1)\phi_j\psi_0+(N-2j+2-1)j\phi_{j-1}\psi_1=(N-2j+1)f_j$$
And
$$R^{N+2}\Omega=(r_N^{N+2}+(N+2)r_N^{N+1}r_1)\phi_0\psi_0=0$$
checks. \bigskip

For the second subrepresentation, take 
\begin{equation}\label{eq:omeganew}
\Omega_{\text{new}}=(r_N-Nr_1)\Omega=\phi_1\psi_0-N\phi_0\psi_1
\end{equation}
and check 
\begin{align*}
L\Omega_{\text{new}}&=(l_N+l_1)(\phi_1\psi_0-N\phi_0\psi_1)=(N+0)\phi_0\psi_0-N(0+1)\phi_0\psi_0=0\\
H\Omega_{\text{new}}&=(N-2+1)\phi_1\psi_0-(N-1)(N\phi_0\psi_1)=(N-1)\Omega_{\text{new}}
\end{align*}
Here the basis takes the form
$$g_j(N-1)=\phi_{j+1}\psi_0-(N-j)\phi_j\psi_1$$
which is readily seen to vanish for $j=N$. \bigskip

It is convenient to use the rescaled bases. Using $(R^j/j!)$ to generate the basis elements we have
\begin{subequations}\label{eq:nminusone}
\begin{align}
\tilde f_j(N+1)&=\tilde\phi_j\psi_0+\tilde\phi_{j-1}\psi_1\\
\tilde g_j(N-1)&=(j+1)\tilde\phi_{j+1}\psi_0-(N-j)\tilde\phi_j\psi_1 \label{eq:nminusoneb}
\end{align}
\end{subequations}

\begin{remark} For orthogonality, as noted above, the states generated by $R$ will be orthogonal as long as
$L$ is adjoint to $R$. Checking $\langle f_j,g_{j-1}\rangle$, one sees directly that the $f_i$'s are orthogonal to the $g_j$'s. For the norms, 
$$\|\tilde f_j\|^2=\binom{N+1}{j}\|\phi_0\|^2$$
consistent with Pascal's triangle, calculating the norm-squared of each term. For $g_j$, first we need
$$\|\Omega'\|^2=\|\phi_1\|^2+N^2\|\phi_0\|^2=N(N+1)\|\phi_0\|^2$$
Then we have
$$\|\tilde g_j\|^2=N(N+1)\binom{N-1}{j}\|\phi_0\|^2$$
which agrees with the calculation using the above expression for $\tilde g_j$.
\end{remark}

\end{subsection}

\section{su(2) content of $\B$}
There are two approaches to detailing the su(2) content of $\B$. 

\begin{subsection}{Global approach}
First we start with a global representation and decompose it.
Set
$$T=\sum_{i=1}^n\hat e_i \,,\qquad T^*=\sum_{i=1}^n\hat \delta_i$$
With $U=[T^*,T]$ the commutator, Proposition \ref{prop:LA} shows that we have a standard triple:
$$[T^*,T]=U\,,\quad [T,U]=2T \,,\quad [U,T^*]=2T^*$$
From the previous section, we see that this is the tensor product of $n$ copies of $\bf 2$. We can find the vacuum states
from a basis for the nullspace of $T^*$ and generate the representation by applying $T$ in each subrepresentation.
We already know the natural basis, $\{e_\I\}$. The goal is to find an orthogonal decomposition into irreducible subrepresentations. \bigskip

Let us define the \textsl{layer operator} or number operator, $\L$, indicating the number of elements of a set:
$$\L e_\I=\ell\,e_\I\,, \qquad \text{ for } e_\I\in\V_\ell$$
We identify it with the matrix $\L$ having entries
$$ \L_{\I\J}=\begin{cases} \ell, & \text{ if } \I=\J \in\B_\ell \\
0,& \text{ otherwise }
\end{cases}
$$
Note that $\L$ satisfies the commutation relations
$$[\L,T]=T \ ,\qquad [T^*,\L]=T^*$$
In terms of the $\hat{}$ operators, acting on $e_\I$ we have the global definition
\begin{equation}\label{eq:layer}
\L=\sum_i\hat e_i\hat \delta_i
\end{equation}
Then equation \eqref{eq:fermi} yields
\begin{equation}\label{eq:nlayer}
\sum\hat \delta_i\hat e_i=\sum (I-\hat e_i\hat \delta_i)=n\,I-\L
\end{equation}

Now consider the action of $U$. Observe that $\V_\ell$ is invariant under $TT^*$ and $T^*T$.

\begin{proposition} \label{prop:js}
On $\V_\ell$, we have
\begin{align*}
TT^*&=\ell \,I+\JS{1}{n\ell} \\
T^*T&=(n-\ell)\,I+\JS{1}{n\ell}
\end{align*}
\end{proposition}
\begin{proof}
\begin{align*}
TT^*&=(\sum\hat e_i)(\sum\hat \delta_j)=\sum_i\hat e_i\hat \delta_i+\sum_{i\ne j}\hat e_i\hat \delta_j\\
T^*T&=(\sum\hat \delta_i)(\sum\hat e_j)=\sum_i\hat \delta_i\hat e_i+\sum_{i\ne j}\hat e_i\hat \delta_j
\end{align*}
Now identify 
$$\JS{1}{n\ell}=\sum_{i\ne j}\hat e_i\hat \delta_j$$
as the operator pairing sets where exactly one element of a given set is replaced to produce the other one. \hfill
\end{proof}

Subtracting, we see the action of $U$:

\begin{corollary} $U$ has entries
$$ U_{\I\J}=\begin{cases} n-2\ell, & \mathrm{ if }\  \I=\J \in\B_\ell  \\
0,& \mathrm{ otherwise }
\end{cases}
$$
\end{corollary}
From the proof we have
$$U=\sum_i\left(\hat \delta_i\hat e_i-\hat e_i\hat \delta_i\right)=\sum_i\hat h_i$$
comparing with equation \eqref{eq:bose}. This verifies that it acts diagonally, constant on layers. 
The relation 
$$ U=n\,I-2\L$$
derived above follows as well by summing over $i$ in \eqref{eq:bose}.
\end{subsection}

\begin{subsection}{Constructive approach}
Now we will  build the representation step by step starting with $\B(1)$, tensoring with 
$\bf 2$ at each stage, with $\psi_0=1$ at all steps, and $\psi_1=e_{n+1}$ at the stage building $\B(n+1)$ from $\B(n)$.
The raising operator $r_N=e_1+\cdots+e_n$ for each component $N\to N\pm1$ in the transition from $\B(n)$ to $\B(n+1)$.
This approach shows us the content --- irreducible components with their multiplicities --- 
and we construct the vacuum states directly at each stage. \bigskip

Fix $n$. The vacuum states for $\B(n)$ are built inductively. The vacuum states are in one-one correspondence with length-$(n+1)$ sequences
of nonnegative integers where the next differs by 1 from the previous one, with initial $0$. For example, $0123212$ for $n=6$. These
correspond to lattice paths in the plane connecting $(0,0)$, $(1,1)$, $(2,N_2)$, $\ldots, (m,N_m),\ldots$ . The numbers in the sequence
$N_0N_1N_2N_3\ldots$ are the successive values of the 
principal number $N$ which change to $N\pm1$ when the representation is tensored with $\bf 2$. The last number in the sequence is the
value of $N$ in the corresponding subrepresentation in $\B(n)$. \bigskip

\begin{remark} This family of lattice paths are called \textit{Dyck paths}. Sometimes Dyck paths must end at level 0. Here we 
need to allow for any final level after $n$ steps. We refer to a transition from some level $N$ to $N+1$ as an \textit{ascent} and a transition
from some level $N$ to $N-1$ as a \textit{descent}.
\end{remark}

Each descent introduces
a factor corresponding to the vacuum state of the subrepresentation $N\to N-1$ at that stage. It has the form
$$\Omega'=(r_N-Nr_1)\Omega$$
where at stage $m$,
$$r_N=e_1+\cdots+e_m\,,\qquad r_1=e_{m+1}\ .$$
Thus, each downward step on such a path from the point $(m,N_m)$ to the next one, $(m+1,N_m-1)$, contributes a factor of
\begin{equation}\label{eq:descentfactor}
e_1+\cdots+e_m-N_me_{m+1}
\end{equation}
with each upward step contributing a factor of 1, leaving the current state unchanged.  \bigskip

If a vacuum state in $\B(n)$
is at level $\ell$, then, on that state, $U=N-2j=N=n-2\ell$ and the successive states of the subrepresentation are of the form $T^j\phi_0$, 
$0~\le~ j~\le~n-2\ell$, ending with $\phi_N$ at level $n-\ell$.  Let $\alpha$ denote the number of descents in the Dyck path labelling a given vacuum state. 
Then the possible values of $N$ are $N=n-2\alpha$, where $0\le\alpha\le n/2$. 
All of the vacua for the subrepresentations having the same $N$ are at level $\alpha$.\bigskip

The raising operator $\displaystyle T=\sum_{1\le i\le n} \hat e_i$ is global, i.e., it is the raising
operator for each of the subrepresentations. Similarly $T^*$ is the global lowering operator. \bigskip

On any subrepresentation with principal number $N$, we have the basis $\phi_j$ with the action, cf. equation \eqref{eq:su2},
\begin{equation}\label{eq:su2action}
T\phi_j=\phi_{j+1}\,,\qquad T^*\phi_j=j(N+1-j)\phi_{j-1}\,,\qquad U\phi_j=(N-2j)\phi_j
\end{equation}

\begin{definition}
We define the
\textit{$Z$-basis} to be the basis generated by the constructive procedure indicated above.
\end{definition}

\begin{example}
The two vacuum states at level 2 for $n=4$ have principal number $N=n-2\ell=0$  with sequence labels and states
\begin{align*}
01210 &\longrightarrow  (e_1+e_2-2e_3)(e_1+e_2+e_3-e_4) \\
01010 & \longrightarrow (e_1-e_2)(e_1+e_2+e_3-e_4)
\end{align*}
For new vacuum states appearing for $\B(n)$,
notice that each descent from $(m,N_m)$ to $(m+1,N_m-1)$ contributes a factor of $e_{m+1}$ to the top term, the one term with $e_n$ as a factor.
Thus, that term is a record of the descents. Referring to the tables  in the Appendix, consider the three new vacuum states at level 2 for $n=5$. We have
top terms $e_2e_5$, $e_3e_5$, and $e_4e_5$ corresponding to the paths/sequences with associated products
\begin{align*}
010121 &\longrightarrow  (e_1-e_2)(e_1+e_2+e_3+e_4-2e_5) \\
012121 &\longrightarrow  (e_1+e_2-2e_3)(e_1+e_2+e_3+e_4-2e_5) \\
012321 &\longrightarrow  (e_1+e_2+e_3-3e_4)(e_1+e_2+e_3+e_4-2e_5) 
\end{align*}
and so on for higher $n$.

\end{example}
\begin{subsubsection}{Number of vacua per layer}
\begin{proposition}There are 
$\displaystyle \binom{n}{\ell}-\binom{n}{\ell-1}$ vacuum states at level $\ell\le n/2$.
\end{proposition}
\begin{proof} Proceeding by induction, for $n=1$, we have $\ell=0$, with the vacuum state $e_\emptyset=1$.
In the transition from $n$ to $n+1$, vacua at layer $\ell$ have $N=n-2\ell$ which go to $N+1=n+1-2\ell$ and
$$N-1=n-1-2\ell=n+1-2(\ell+1)\ .$$
Thus, vacua at layer $\ell$ come from layers $\ell$ and $\ell-1$ of $\B(n)$. So we must check the relation
\begin{align*}
\binom{n+1}{\ell}-\binom{n+1}{\ell-1}&=\binom{n}{\ell}-\binom{n}{\ell-1}+\binom{n}{\ell-1}-\binom{n}{\ell-2}\\
&=\binom{n}{\ell}+\binom{n}{\ell-1}-\left[\binom{n}{\ell-1}+\binom{n}{\ell-2}\right]\\
&=\binom{n+1}{\ell}-\binom{n+1}{\ell-1}
\end{align*}
as required, via Pascal's triangle.
\end{proof}
\end{subsubsection}
\begin{subsubsection}{Orthogonality}
We will check that the $Z$-basis is orthogonal, thus, by scaling, the transformation from the natural basis to the $Z$-basis is unitary. \bigskip

First, note that the operator $U$ shows that states from different layers are orthogonal. Next, by construction, vacuum states in a given layer
are orthogonal. We compare states in two different irreducible subrepresentations as follows. Note that if $\Omega$ is a vacuum state with
principal number $N$, then
$$e^{tT}\Omega=\sum_{j=0}^N \frac{t^j}{j!}\, \phi_j$$
for the corresponding states $\phi_j$. We refer to the Leibniz Rule, proved below, which provides the relation
$$e^{tT^*}e^{aT}=\exp\left(\frac{aT}{1+at}\right)\,(1+at)^U\,\exp\left(\frac{tT^*}{1+at}\right)$$
So, consider, for $\Omega,\Omega'$ two vacua in layer $\ell$,
\begin{align*}
\langle e^{tT}\Omega,e^{sT}\Omega'\rangle&=\langle \Omega,e^{tT^*}e^{sT}\Omega'\rangle\\
&=\langle\Omega,\exp\left(\frac{sT}{1+st}\right)\,(1+st)^U\,\Omega' \rangle\\
&=(1+st)^{n-2\ell}\,\langle\Omega,\Omega'\rangle\\
\end{align*}
since, by moving across the inner product,  the exponential-$T$ terms as well as the exponential-$T^*$ terms act like the identity on vacua.
Since the result is a function of the product $st$, the corresponding states for differing values of $j$ are orthogonal. And for states in
different subrepresentations, the orthogonality of $\Omega$ and $\Omega'$ yield orthogonality of all states in different subrepresentations.
For the inner product of $\phi_j$ with itself, we recover equation \eqref{eq:ortho1}.
\end{subsubsection}\bigskip

We make the observation that the operators $$\{TT^*,T^*T,U,\L,T^2{T^*}^2,\ldots,T^k{T^*}^k,\ldots\}$$ act diagonally on the $Z$-basis. 
Hence they generate a commutative algebra with each $\V_\ell$ an invariant subspace.
\end{subsection}\bigskip

See the charts in the Appendix for values of $TT^*$ and for the explicit form of the states for values of $n$ from 1 to 5. The states in these
charts are $\tilde\phi_j=\phi_j/j!$ in each subrepresentation, i.e., $\tilde\phi_j=(T^j/j!)\phi_0$.

\begin{subsection}{Labelling the states}
We have two basic ways to look at the su(2) states in $\B$. One set of labels $|nNij\rangle$ denotes the $j^{\rm th}$ state in the
$i^{\rm th}$ column (of those) with principal number $N$, counting from the left, in $\B(n)$. \bigskip

Another set of labels is $|n\ell k\rangle$, which indicates the $k^{\rm th}$ state from the left in layer $\ell$ in $\B(n)$. We are 
numbering the layers starting with $0$ at the top to $n$ at the bottom.

\begin{figure}[ht!]
\begin{center}
\scalebox{.5}{\input{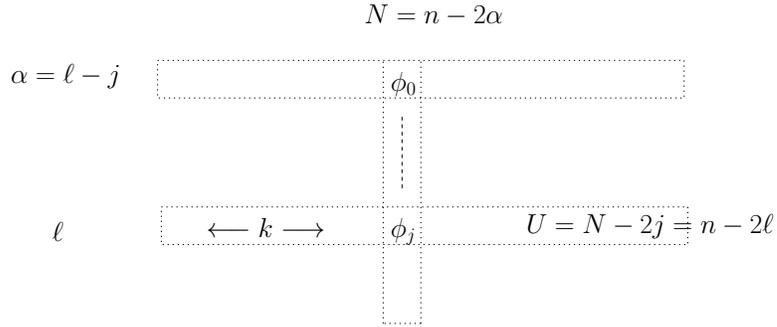}}
\caption{State labels: $|nNij\rangle=|n\ell k\rangle$}\label{fig:labelling}
\end{center}
\end{figure}

Given $N$, $i$, and $j$, find $\ell$ from the fact that $U$ acts globally, so that at level $j$ in any subrepresentation with principal number $N$, we have
$$U=n-2\ell=N-2j\ .$$
With $N=n-2\alpha$, the above relation yields $\alpha=\ell-j$ corresponding to  blocks of subrepresentations with the same principal number $N$. 
The number $\alpha$ is thus the level of the vacuum state for the representation from the point of view of the state $\phi_j$. Then 
\begin{align*}
&&&k=i+\binom{n}{\alpha-1}\,,&&\text{where $i$ runs from $1$ to }\binom{n}{\alpha}-\binom{n}{\alpha-1}\ .
\end{align*}
Note that there are $\displaystyle \binom{n}{\ell}-\binom{n}{\ell-1}$ vacua in layer $\ell\le n/2$, with $\alpha=\ell$ for vacuum states. \bigskip

Given $\ell$ and $k$, we have
$$\alpha=\max\{\alpha'\colon \binom{n}{\alpha'-1}< k\}$$
And
$$i=k-\binom{n}{\alpha-1}\,,\qquad N=n-2\alpha\,,\qquad j=\ell-\alpha$$
fill out the corresponding state label.
\bigskip
\end{subsection}

\begin{subsection}{$\alpha$-chains}
Now, start with a vacuum state $\phi_0$ at level 
$$\ell_{\text{initial}}=\alpha=\half(n-N)\ .$$
Then the successive states run through levels $\alpha+1$ to $\alpha+N$. At that level, $j=N$, and
$$U=N-2j=-N=n-2\ell_{\text{final}}$$
or 
$$\ell_{\text{final}}=\half\,(n+N)=n-\ell_{\text{initial}}\ .$$
So the states form a chain symmetric about the middle, corresponding to the layers. \bigskip

Define a partial order on the $Z$-basis by
$$\phi \le \psi \quad\text{if}\quad \psi=T^k\phi$$ 
for some $k\ge0$. 
The states for an irreducible subrepresentation form a chain.
It is convenient to refer to the set of states satisfying $\phi_{j+1}=T\phi_j$, with $T^*\phi_0=0$, $\phi_0$ in level $\alpha$, as an 
\textit{$\alpha$-chain}. These chains are transversal to the layers. 

\begin{subsubsection}{Complementation}
Looking at the charts in the Appendix, one observes that states in the same chain on complementary levels $\ell\leftrightarrow n-\ell$ are complementary
states up to sign, we will show that, in fact,
\begin{proposition}
$$\tilde\phi_j'=(-1)^\alpha \tilde\phi_{N-j}$$
\end{proposition}
The proof  will emerge in the following discussion. \bigskip

The action of $T$ and $T^*$ on the scaled states is
$$T\tilde\phi_j=(j+1)\tilde\phi_{j+1} \qquad \text{and}\qquad T^*\tilde\phi_{N-j}=(j+1)\tilde\phi_{N-j-1}$$
where it is convenient to look at the action of $T^*$ from the reverse point of view.  \bigskip

First, we extend eq.\,\eqref{eq:comp} to $U$:
\begin{align*}\label{eq:ucomp}
U\phi'&=(T^*T-TT^*)\phi'=T^*(T^*\phi)'-T(T\phi)'\\
&=(TT^*\phi)'-(T^*T\phi)'\\
&=-(U\phi)'\ .
\end{align*}
In a given $\alpha$-chain, the states are uniquely identified up to scale by the eigenvalue for $U$. And we have
$$U\phi_j'=-(U\phi_j)'=-(N-2j)\phi_j'=(N-2(N-j))\phi_j'$$
i.e., $\phi_j'$ is an eigenvector of $U$ with eigenvalue the same as $\phi_{N-j}$. Hence they differ at most by a scalar.
In particular, for $j=0$, let
\begin{equation}\label{eq:priming}
\phi_0=\epsilon\, \tilde\phi_N'
\end{equation}
Now,
\begin{align*}
\tilde\phi_j'&=((T^j/j!)\phi_0)'=((T^*)^j/j!)\phi_0'=\epsilon\,((T^*)^j/j!)\tilde\phi_N\\
&=\epsilon\,\tilde\phi_{N-j}\ .
\end{align*}
In other words, the states all have the same scalar factor
\begin{equation}\label{eq:epsilon}
\tilde\phi_j'=\epsilon\, \tilde\phi_{N-j}
\end{equation}
First, let's see that $\epsilon=\pm1$. Let $j=N$ in eq.\,\eqref{eq:epsilon}
$$\tilde\phi_N'=\epsilon\, \phi_0$$
Substitute into \eqref{eq:priming} to get
$$\phi_0=\epsilon^2 \phi_0$$
i.e., $\epsilon=\pm1$. \bigskip

Next, we  check that for every descent in the Dyck path corresponding to a vacuum state, the factor $\epsilon$ changes sign.
Recall equations \eqref{eq:omeganew} and \eqref{eq:descentfactor}. So a typical new vacuum looks like
$$\Omega_{\text{new}}=\phi_1-N\phi_0\,e_{m+1}$$
By induction, we have $\phi_0'=\epsilon\,\tilde\phi_N$ and $\phi_1'=\epsilon\, \tilde\phi_{N-1}$. Hence
$$\Omega_{\text{new}}'=\epsilon(\tilde\phi_{N-1}e_{m+1}-N\tilde\phi_N)=-\epsilon\,\tilde\phi_{N-1,\text{new}}$$
by eq.\,\eqref{eq:nminusoneb} with $\psi_1=e_{m+1}$ and $j=N-1$. Starting with $n=1$, $\epsilon=1$ yields the Proposition. \bigskip

\end{subsubsection}
\end{subsection}
\begin{subsection}{References}
See \cite[p.\,216]{BI} and \cite[p.\,174]{SST} for another approach to vacuum states. Also, \cite[\S\S 7.9ff]{GSL} for yet another approach.
\end{subsection}
\begin{subsection}{Notes on the Appendix}
As noted above, the first set of tables, called Boolean System, for $n=1$ through $n=5$, show the values of $TT^*$ and give the states in terms of the $e_i$ variables.
A table of values of the Casimir shows the overall structure for different values of $N$.
The second set of tables, called Boolean States, for $n=2$ through $n=6$, shows just the coefficients, i.e., the coordinates of the corresponding
states. For each level $\ell$, reading left to right in the corresponding Boolean System table, each state is written as a row in the corresponding matrix.
Since the states are not normalized, a second matrix is given with the squared norms of the rows. Consider $n=3$, level 2, the matrix of states
and matrix of squared norms are
$$ \begin{pmatrix}1&1&1\\ 0&1&-1\\ 2&-1&-1 \end{pmatrix} \;,\qquad \begin{pmatrix}3&0&0\\ 0&2&0\\ 0&0&6 \end{pmatrix}$$
For example, the middle row of the first matrix is the state $e_1e_3-e_2e_3$. 
For a given level, call the matrix of states $W$ and the diagonal matrix of squared norms $D$.
Then
$$ V=WD^{-1/2}$$
is a unitary matrix. The collection of such $V$'s provides the unitary transformation from the natural basis to an orthonormal system for $\B(n)$.
\end{subsection}

\section{Group elements}

For any nonzero number, $u$, we define $u^{\L}$ by its action on level $\ell$
$$
u^{\L} e_\I=u^\ell\,e_\I \,,\qquad \text{if } |\I|=\ell
$$

The group element $g(s,u,t)$ is defined as
$$g(s,u,t)=e^{sT}\,u^{\L}\,e^{tT^*}$$

\begin{theorem}\label{thm:grpels}
For the group elements we have
$$ (e^{sT}\,u^{\L}\,e^{tT^*})_{\I\J}=s^{|\I\,\setminus\, \J|}\,(u+st)^{|\I\cap \J|}\,t^{|\J\,\setminus\, \I|}$$
\end{theorem}
\begin{proof}
A typical term in the expansion of the group element has the form
$$ s^a\,(T^a/a!)\,u^{\L}\,t^b\,((T^*)^b/b!)$$
\begin{figure}[ht!]
\begin{center}
\scalebox{.5}{\input{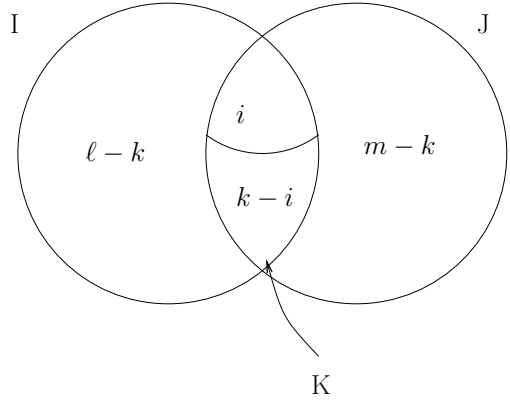}}
\caption{Configuration of sets for finding the group elements}\label{fig:grpels}
\end{center}
\end{figure}
Starting with $\I$ in layer $\ell$, we have a set $\K$ such that $\I\supset \K\subset \J$, i.e., $\K\subset \I\cap\J$.
Let $k=|\I\cap\J|$, $k-i=|\K|$. Then $a$, the difference in size between $\I$ and $\K$ must equal $\ell-k+i$ and, 
similarly, with $\J$ in layer $m$, $b=m-k+i$. There are $\displaystyle \binom{k}{i}$ such choices of $\K$ in layer $k-i$. So the terms contribute
$$\sum_i \binom{k}{i}\,s^{\ell-k+i}\,u^{k-i}\,t^{m-k+i}=s^{\ell-k}(u+st)^kt^{m-k}$$
Comparing with the figure shows that this agrees with the stated formula.
\end{proof}

\begin{subsection}{Specializations: Johnson scheme, Boolean poset, Hadamard-Sylvester matrices}
\begin{subsubsection}{Johnson scheme}
First, let $u=1$, $s=t$. Theorem \ref{thm:grpels} provides the formula
$$ (e^{tT}\,e^{tT^*})_{\I\J}=t^{|\I\,\DD\, \J|}\,(1+t^2)^{|\I\cap \J|}$$
Restricting to layer $\ell$,  if $\text{dist}_{J}(\I,\J)=j$, then $|\I\cap \J|=\ell-j$ and $|\I\,\DD\, \J|=2j$. So, on $\V_\ell$, we have the matrix elements equal to
$$t^{2j}(1+t^2)^{\ell-j}$$
Specialize to $t=1$. And we have, on $\V_\ell$,
$$2^{|\I\,\cap\, \J|}=2^{\ell-j}\,\JS{j}{n\ell}$$
That is, restricting to $|\I|=|\J|=\ell$, we have
$$e^{T}\,e^{T^*}\Bigr|_{\V_\ell}=\sum_{0\le j\le \min(\ell,n-\ell)} 2^{\ell-j}\,\JS{j}{n\ell}$$
Notice that this is a binary expansion with coefficients the Johnson matrices. So having computed $e^Te^{T^*}$, the Johnson matrices can quickly be
found by the same way the binary expansion of a number is found. First restrict to the $\binom{n}{\ell}\times\binom{n}{\ell}$ central submatrix
of $e^Te^{T^*}$ corresponding to $\V_\ell$. Now, iterate the steps: reduce mod 2, save the result, subtract it off, divide by 2, 
reduce mod 2, save the result, subtract, and so on. For $0\le\ell\le n/2$, the successive results are $\JS{\ell}{n\ell}$, $\JS{\ell-1}{n\ell}$, etc., 
to $\JS{0}{n\ell}$. Thus the Johnson basis is produced.  The procedure works as well for all $0\le \ell\le n$.\bigskip

\begin{example} Take $n=4$. We have
$$e^Te^{T^*}=
\left(\begin {array}{cccccccccccccccc} 1&1&1&1&1&1&1&1&1&1&1&1&1&1&1
&1\\\noalign{\medskip}1&2&1&1&1&2&2&2&1&1&1&2&2&2&1&2
\\\noalign{\medskip}1&1&2&1&1&2&1&1&2&2&1&2&2&1&2&2
\\\noalign{\medskip}1&1&1&2&1&1&2&1&2&1&2&2&1&2&2&2
\\\noalign{\medskip}1&1&1&1&2&1&1&2&1&2&2&1&2&2&2&2
\\\noalign{\medskip}1&2&2&1&1&4&2&2&2&2&1&4&4&2&2&4
\\\noalign{\medskip}1&2&1&2&1&2&4&2&2&1&2&4&2&4&2&4
\\\noalign{\medskip}1&2&1&1&2&2&2&4&1&2&2&2&4&4&2&4
\\\noalign{\medskip}1&1&2&2&1&2&2&1&4&2&2&4&2&2&4&4
\\\noalign{\medskip}1&1&2&1&2&2&1&2&2&4&2&2&4&2&4&4
\\\noalign{\medskip}1&1&1&2&2&1&2&2&2&2&4&2&2&4&4&4
\\\noalign{\medskip}1&2&2&2&1&4&4&2&4&2&2&8&4&4&4&8
\\\noalign{\medskip}1&2&2&1&2&4&2&4&2&4&2&4&8&4&4&8
\\\noalign{\medskip}1&2&1&2&2&2&4&4&2&2&4&4&4&8&4&8
\\\noalign{\medskip}1&1&2&2&2&2&2&2&4&4&4&4&4&4&8&8
\\\noalign{\medskip}1&2&2&2&2&4&4&4&4&4&4&8&8&8&8&16\end {array}
 \right)$$
The block for $\ell=2$ is
$$\left( \begin {array}{cccccc} 4&2&2&2&2&1\\\noalign{\medskip}2&4&2&2&
1&2\\\noalign{\medskip}2&2&4&1&2&2\\\noalign{\medskip}2&2&1&4&2&2
\\\noalign{\medskip}2&1&2&2&4&2\\\noalign{\medskip}1&2&2&2&2&4
\end {array} \right)$$
and the Johnson basis is immediate.
\end{example}
\end{subsubsection}

\begin{subsubsection}{Boolean poset}
Now specialize to $u=1$, $s=0$, $t=\pm1$. For $t=1$, we have the exponential of $T^*$. From equation \eqref{eq:inclusion} as well as the Theorem,
we have, calling $e^{T^*}$, $E$  in the present context,
$$E_{\I\J}=(e^{T^*})_{\I\J}=\begin{cases} 1, & \text{ if } \I\subset \J \\ 0,& \text{ otherwise }\end{cases}$$
That is, $E$ is the incidence matrix for the Boolean poset $\B(n)$. We immediately have the M\"obius matrix
$$M=E^{-1}=e^{-T^*}$$
For $v\in\V_\ell$, we have, for general $t$,
\begin{align*}
(Ev)_\I=(e^{tT^*}v)_\I&=\sum_{J\supset\I}t^{|\J\setminus\I|}v_\J\\
(Mv)_\I=(e^{-tT^*}v)_\I&=\sum_{J\supset\I}(-1)^{|\J\setminus\I|}t^{|\J\setminus\I|}v_\J
\end{align*}
which for $t=1$ gives the M\"obius inversion for the Boolean poset.  Many counting formulas, in particular
inclusion-exclusion, may be readily derived from this relation. Cf. \cite[p.\,237ff]{SST}.
\end{subsubsection}

\begin{subsubsection}{Sylvester-Hadamard matrices}
Recall the Sylvester-Hadamard matrices. They are an infinite family of Hadamard matrices built as follows.
Start with 
$$H_1=\begin{pmatrix} \hfill 1&\hfill 1\\  \hfill 1&-1\end{pmatrix}$$
then iteratively
$$H_{n+1}=H_1\otimes H_n$$
i.e., iterated Kronecker products of matrices, conventionally associating to the left. For example,
$$H_2=\left( \begin {array}{rrrr} 1&1&1&1\\\noalign{\medskip}1&-1&1&-1
\\\noalign{\medskip}1&1&-1&-1\\\noalign{\medskip}1&-1&-1&1\end {array}
 \right) \,,\quad H_3=
 \left( \begin {array}{rrrrrrrr} 1&1&1&1&1&1&1&1\\\noalign{\medskip}1&
-1&1&-1&1&-1&1&-1\\\noalign{\medskip}1&1&-1&-1&1&1&-1&-1
\\\noalign{\medskip}1&-1&-1&1&1&-1&-1&1\\\noalign{\medskip}1&1&1&1&-1&
-1&-1&-1\\\noalign{\medskip}1&-1&1&-1&-1&1&-1&1\\\noalign{\medskip}1&1
&-1&-1&-1&-1&1&1\\\noalign{\medskip}1&-1&-1&1&-1&1&1&-1\end {array}
 \right) $$

Clearly $H_n$ is $2^n\times 2^n$. Now label rows and columns of $H_1$ with bits. So
$$1=(H_1)_{00}=(H_1)_{10}=(H_1)_{01}=-(H_1)_{11}$$
Or, using $\alpha$, $\beta$ to denote the corresponding bit strings,
$$(H_1)_{\alpha\beta}=(-1)^{|\alpha\cap \beta|}=(-1)^{\alpha\cdot\beta}$$
with the $\cap$ denoting intersection of the corresponding sets and the dot denoting ordinary dot product.

The difference with our approach is that the iterated Kronecker products are indexed by binary strings, i.e., the subsets
are identified with their corresponding bit strings or bit vectors, whereas we have been ordering layer-by-layer. If we are given $X$
and take the product with $H_1$ the matrix elements are
$$(H_1\otimes X)_{i \alpha, j\beta}=(H_1)_{ij}X_{\alpha\beta}$$
so the value flips whenever $i=j$. So inductively we have
$$(H_n)_{\alpha\beta}=(-1)^{|\alpha\cap \beta|}=(-1)^{\alpha\cdot\beta}$$
for all $n\ge1$. \bigskip

Taking $u=-2$, $s=t=1$, we have
\begin{align*}
(e^T(-2)^{\L} e^{T^*})_{\I\J}&=1^{|\I\,\setminus\, \J|}\,(-1)^{|\I\cap \J|}\,1^{|\J\,\setminus\, \I|}\\
&=(-1)^{|\I\cap \J|}
\end{align*}
which is thus permutation equivalent to the corresponding Sylvester-Hadamard matrix.  \bigskip

Referring to the group law, Proposition \eqref{prop:grouplaw}, \S\ref{sec:grouplaw}, we have
the group element $g(s,u,t)=g(1,-2,1)$. By construction it is a symmetric matrix, so to check that it is Hadamard, we calculate its square.
Using the group law, we have
$$g(1,-2,1)g(1,-2,1)=2^n\,g(1+(-2/2),4/4,1+(-2/2))=2^n\,g(0,1,0)=2^n I$$
as expected.
\end{subsubsection}
\end{subsection}

\begin{subsection}{Leibniz rule}
The commutation rule moving $T^*$ past $T$ is an analog of the Leibniz rule in calculus commuting the operator of differentiation
past that of multiplication by $x$, i.e. moving the lowering operator past the raising operator. For su(2) this rule takes the form
stated in the following proposition.
\begin{proposition} {\sl Leibniz rule}\\
The matrix elements of both sides of
$$e^{tT^*}e^{aT}=\exp\left(\frac{aT}{1+at}\right)\,(1+at)^U\,\exp\left(\frac{tT^*}{1+at}\right)$$
are
$$(e^{tT^*}e^{aT})_{\I\J}=a^{|\I\,\setminus\,\J|}\,(1+at)^{|(\I\,\cup\,\J)'|}\,t^{|\J\,\setminus\,\I|}$$
the prime indicating complement.
\end{proposition}
\begin{figure}[ht!]
\begin{center}
\scalebox{.4}{\input{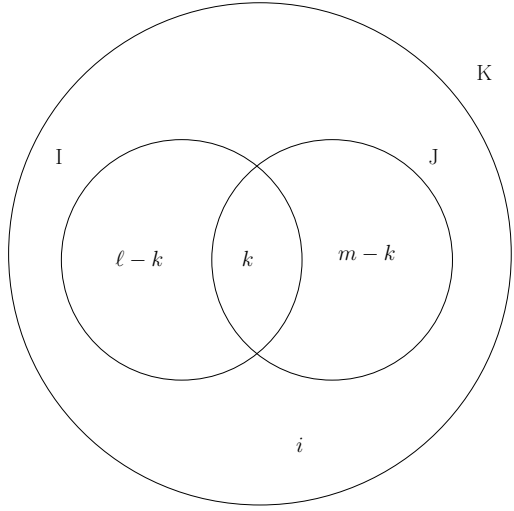}}
\caption{Configuration of sets for the Leibniz rule}\label{fig:leibniz}
\end{center}
\end{figure}
\begin{proof}
Start with $\I$ in layer $\ell$ and $\J$ in layer $m$ with $|\I\cap\J|=k$. Then corresponding to a typical term
$$t^\alpha((T^*)^\alpha/\alpha!)\,a^\beta (T^\beta/\beta!)$$
is a set $\K$ such that $\I\subset\K\supset\J$, i.e., $\K\supset  \I\cup\J$. Let $|\K|=l+m-k+i$.
From the figure, we see that
$$\alpha=|\K\,\setminus\,\I|=m-k+i\,,\qquad \beta=|\K\,\setminus\,\J|=\ell-k+i$$
There are $\displaystyle \binom{n+k-\ell-m}{i}$ choices of such a $\K$, providing a contribution of
$$\sum_i\binom{n+k-\ell-m}{i}\,t^{m-k+i}\,a^{\ell-k+i}=a^{\ell-k}\,(1+at)^{n-\ell-m+k}\,t^{m-k}$$
as required. For the right-hand side, we replace $U=n-2\L$ and use Theorem \ref{thm:grpels}. In the formula for the group elements,
we have on the left-hand side, using primes to denote the coordinates,
$$s'\leftarrow a/(1+at)\,,\qquad u'\leftarrow 1/(1+at)^2\,,\qquad t'\leftarrow t/(1+at)$$
Thus,
\begin{align*}
(1+at)^n\,s'^{|\I\,\setminus\, \J|}\,((1&+at)^{-2}+s't')^{|I\cap J|}\,t'^{|\J\,\setminus\, \I|}\\
&=(1+at)^{n-2|I\cap \J|}\,s'^{|\I\,\setminus\, \J|}\,(1+s't'(1+at)^2)^{|I\cap \J|}\,t'^{|\J\,\setminus\, \I|}\\
&=(1+at)^{n-2|I\cap \J|-|\I\,\setminus\, \J|-|\J\,\setminus\, \I|}\,a^{|\I\,\setminus\,\J|}\,(1+at)^{|\I\,\cap\,\J|}\,t^{|\J\,\setminus\,\I|}\\
&=(1+at)^{n-|I\cap \J|-|\I\,\setminus\, \J|-|\J\,\setminus\, \I|}\,a^{|\I\,\setminus\,\J|}\,t^{|\J\,\setminus\,\I|}
\end{align*}
which agrees with the left-hand side.
\end{proof}

\end{subsection}

\begin{subsection}{Group law}\label{sec:grouplaw}
Now we can find the group law. First,
\begin{proposition} \label{prop:scaling}
For any polynomial or exponential function $f$, we have, for nonzero $a$,
$$ a^{\L} f(T)=f(aT)a^{\L}\,,\qquad f(T^*)a^{\L}=a^{\L} f(aT^*) $$
\end{proposition}
\begin{proof} From the definitions of $\L$ as the number operator and $T$ as raising operator we have the implications
$$ \L T=T(\L+1)\Longrightarrow \L^n T=T(\L+1)^n \Longrightarrow a^{\L} T=Ta^{\L+1}\Longrightarrow a^{\L} T^n=(aT)^na^{\L}$$
which yields the first relation. The second follows similarly.
\end{proof}

\begin{proposition} \label{prop:grouplaw} The group law is
$$g(a,c,b)g(s,u,t)=(1+sb)^n\,g\left( a+\frac{sc}{1+sb},\frac{uc}{(1+sb)^2},t+\frac{ub}{1+sb}\right)$$
\end{proposition}
\begin{proof}
The essential part is
$$c^{\L} \,e^{bT^*}e^{sT}\,u^{\L}=(1+sb)^n\,c^{\L}\exp\left(\frac{sT}{1+sb}\right)\,\bigl((1+sb)^{-2}\bigr)^{\L}\,\exp\left(\frac{bT^*}{1+sb}\right)u^{\L}$$
via the Leibniz rule. Now use Proposition \ref{prop:scaling} to get the middle factor. Padding with the first and last factors, one arrives at the required formula.
\end{proof}
\end{subsection}

\begin{subsection}{Group elements via exponentiation}
A typical element of the Lie algebra has the form $sT+vU+tT^*$. Since $U$ acts diagonally on the $Z$-basis, we restrict to the form
$sT+tT^*$. We have another reduction by the relation
$$sT+tT^*=\sqrt{st}\,\left({\textstyle\sqrt{\frac{s\mathstrut}{t}}}\,T+{\textstyle\sqrt{\frac{t}{s}}}\,T^*\right)$$
and observing that this is of the form $a(T'+{T'}^*)$, where $T'=\sqrt{s/t}\,T$ and ${T'}^*=\sqrt{t/s}\,T^*$ with $[{T'}^*,T']=U$  
form a standard triple. In other words, up to scaling, we need only consider
exponentiation of $t(T+T^*)$. First we state the result, then discuss proofs.

\begin{theorem} \label{thm:exp} We have
$$\exp\bigl(t\,(T+T^*)\bigr)=e^{(\tanh t)T}\,(\cosh t)^U\,e^{(\tanh t)T^*}$$
\end{theorem}
This result is known in a variety of contexts, from the study of symmetric spaces to systems of ordinary differential equations. Here we wish to provide
a proof in the context of the zeon algebra, using properties from Section \S\ref{subsec:zeon}. \bigskip

Writing $T+T^*=\sum_i (\hat e_i+\hat \delta_i)$, we see that it is sufficient to consider the single pair $\hat e_1$ and $\hat \delta_1$.
Consider $P=\hat \delta_1\hat e_1$ and $Q=\hat e_1\hat \delta_1$. The lemma only depends on their properties as complementary
projections.

\begin{lemma} Let $P$ and $Q$ be idempotents satisfying $P+Q=I$, $PQ=QP=0$. Then \medskip

1. $e^{tP}=Q+e^t\,P$ . \par
2. For $a>0$, 
$$ a^{P-Q}=aP+a^{-1}Q$$
\end{lemma}
\begin{proof} First,
$$ e^{tP}=I+P(e^t-1)=I-P+e^t\,P=Q+e^t\,P\ .$$
Next, write $a=e^t$. Then using \#1 for $P$, then $Q$, we have
\begin{align*}
a^{P-Q}&=e^{tP}e^{-tQ}\\
&=(Q+aP)(P+a^{-1}Q)\\
&=aP+a^{-1}Q
\end{align*}
as required.
\end{proof}
Of course, $e^{t\hat e_1}=I+t\hat e_1$, and similarly for $\hat \delta_1$.  And observe that $\hat e_1+\hat \delta_1$ squares to the identity
$$(\hat e_1+\hat \delta_1)^2=\hat e_1\hat \delta_1+\hat \delta_1\hat e_1=I$$
Then 
$$\exp\bigl(t(\hat e_1+\hat \delta_1)\bigr)=I\,\cosh t+(\hat e_1+\hat \delta_1)\sinh t$$
follows immediately. This is the left-hand side of the formula for $n=1$. Now we verify the right-hand side, using $P$ and $Q$ as above, we recall
equation \eqref{eq:triples}, so that $\hat e_1 P=\hat e_1$ and $P\hat \delta_1=\hat \delta_1$, the reversed products vanishing. So
\begin{align*}
&(I+\hat e_1\tanh t)\,(\cosh t)^{P-Q}\,(I+\hat \delta_1\tanh t)\\
&=(I+\hat e_1\tanh t)\,(P\cosh t+Q\,\mathrm{sech}\, t)\,(I+\hat \delta_1\tanh t)\\
&=(I+\hat e_1\tanh t)(P\cosh t+Q\,\mathrm{sech}\, t+\hat \delta_1\sinh t)\\
&=P\cosh t+Q\,\mathrm{sech}\, t+\hat \delta_1\sinh t+\hat e_1\sinh t +Q\sinh t\tanh t\\
\end{align*}
Using the elementary identity $\mathrm{sech}\, t+\sinh t\tanh t=\cosh t$, this last reduces accordingly.
\end{subsection}

\section{Krawtchouk polynomials and the Hamming scheme. The Johnson scheme}
\subsection{Krawtchouk polynomials and the Hamming scheme}
Combining Theorems \ref{thm:grpels} and \ref{thm:exp}, we have the matrix elements
\begin{align}
(e^{t\,(T+T^*)})_{\I\J}&=\bigl((\cosh t)^n\,e^{(\tanh t)T}\,(\mathrm{sech}^2t)^{\L}\,e^{(\tanh t)T^*}\bigr)_{\I\J} \nonumber   \\
&=(\cosh t)^n\,(\tanh t)^{|\I\DD\J|}  \label{eq:krav}
\end{align}
starting with the relation $U=nI-2\L$ and noting that with the substitutions $s=t=\tanh t$, $u=\mathrm{sech}^2t$, 
we have $u+st=\mathrm{sech}^2t+\tanh^2t=1$. Now let $v=\tanh t$, we have $\mathrm{sech}\, t=\sqrt{1-v^2}$. \bigskip

\begin{notation}   For the remainder of this section we will write $X$ for $T+T^*$.
\end{notation}

And eq. \eqref{eq:krav} becomes
$$e^{tX}=(1-v^2)^{-n/2}\,v^{|\I\DD\J|}$$
Furthermore, we can solve for $e^t$:
$$e^t=\sqrt{\frac{1+v}{1-v}}$$
and rearrange to get
\begin{equation}\label{eq:krav2}
\bigl((1+v)^{(n+X)/2}\,(1-v)^{(n-X)/2}\bigr)_{\I\J}=v^{|\I\DD\J|}
\end{equation}
On the left-hand side we recognize the generating function for the \textit{Krawtchouk polynomials}:
$$(1+v)^{(n+X)/2}\,(1-v)^{(n-X)/2}=\sum_j v^j K_j(X,n)/j!$$
in a convenient scaling, and expanding the binomials yields the explicit form
$$K_j(X,n)/j!=\sum_i \binom{(n+X)/2}{j-i}\binom{(n-X)/2}{i}\,(-1)^i$$
Thus,
\begin{theorem} We have the global relation
$$K_j(T+T^*,n)/j!=\HS{j}{n}$$
\end{theorem}
\begin{proof} 
Write eq. \eqref{eq:krav2} as
$$\sum_j v^j \left(K_j(X,n)\right)_{\I\J}/j!=v^{|\I\DD\J|}$$
Recalling the Hamming matrix from equation \eqref{eq:hamming}, we have
$$\sum v^j (\HS{j}{n})_{\I\J}=v^{|\I\DD\J|}$$
which is exactly the relation above. \hfill\qedhere
\end{proof}

Now we work within a single subrepresentation, $\alpha$-chain, with principal number $N$. Applying $e^{tX}$ to a vacuum state
$\phi_0(N)$, we have
$$e^{tX}\phi_0(N)=e^{(\tanh t)T}\,(\cosh t)^U\,e^{(\tanh t)T^*}\phi_0(N)=(\cosh t)^N\,e^{(\tanh t)T}\phi_0(N)$$
via $U\phi_0(N)=N\phi_0(N)$. Proceeding as above, substituting $v=\tanh t$, comparing eqs. \eqref{eq:krav} and \eqref{eq:krav2} with the
equation above yields the local formula
\begin{align*}
(1+v)^{(N+X)/2}\,(1-v)^{(N-X)/2}\phi_0(N)&=\bigl(\sum_j v^j K_j(X,N)/j!\bigr)\,\phi_0(N)\\
&=e^{vT}\phi_0(N)=\sum_j v^j \phi_j(N)/j!
\end{align*}
and hence the identification
\begin{theorem} For each $\alpha$-chain, with $N=n-2\alpha$, we have the states
$$\phi_j(N)=T^j\phi_0(N)=K_j(T+T^*,N)\,\phi_0(N)$$
\end{theorem}
In this connection, we see from another point of view the recurrence relation for the Krawtchouk polynomials as precisely the action
of $T+T^*$ on the $Z$-basis:
$$X\phi_j=(T+T^*)\phi_j=\phi_{j+1}+j(N+1-j)\,\phi_{j-1}$$
according to the su$(2)$ action, equation \eqref{eq:su2action}.

\subsection{Johnson scheme}
Here, we start with the Hermitian matrix, for $u$ real,
$$g(z,u,\bar z)=e^{zT}u^{\L} e^{\bar z T^*}$$
We use a variation of the \textit{Bargmann transform}, which is a complex version of a two-dimensional Gaussian integral. We write
$$dz\,d\bar z=(dx+i\,dy)(dx-i\,dy)=-2i\, dx\,dy$$
and recall the Gaussian integral
$$\iint\limits_{\mathbb{R}^2} e^{-(x^2+y^2)/2} \,dx\,dy=2\pi$$
For us, the Bargmann transform takes the form
$$\iint e^{-z\bar z}f(z,\bar z)\,{\textstyle\frac{i\,dz\,d\bar z}{2\pi}}=\int_0^\infty\!\!\int_0^{2\pi}e^{-r^2}f(re^{i\theta},re^{-i\theta})\,2r\,dr\,d\theta/(2\pi)$$
changing to polar coordinates.  \bigskip

First, we apply this to the matrix elements of $g(z,u,\bar z)$. We have, from Theorem \ref{thm:grpels}, with $r^2=z\bar z$,
\begin{align*}
\iint e^{-z\bar z}g(z,u,\bar z)\,{\textstyle\frac{i\,dz\,d\bar z}{2\pi}}
&=\iint e^{-z\bar z}\,z^{|\I\,\setminus\, \J|}\,(u+r^2)^{|\I\cap \J|}\,{\bar z}^{|\J\,\setminus\, \I|}\,{\textstyle\frac{i\,dz\,d\bar z}{2\pi}}\\
&=\int_0^\infty\!\!\int_0^{2\pi}e^{-r^2} r^{|\I\,\DD\, \J|}\,e^{i\theta(|\I\,\setminus\, \J|-|\J\,\setminus\, \I|)}  \,(u+r^2)^{|\I\cap \J|}\,2r\,dr\,d\theta/(2\pi)\\
\end{align*}
Integrating over $\theta$ forces $|\I\,\setminus\, \J|=|\J\,\setminus\, \I|$, or $|\I|=|\J|$.
So take $\I,\J\in\B_\ell$ with $\text{dist}_{J}(\I,\J)=k$. We have
\begin{align*}
\int_0^\infty e^{-r^2} r^{2k}(u+r^2)^{\ell-k}\,2r\,dr&=\int_0^\infty e^{-x} x^k (u+x)^{\ell-k}\,dx
&\text{ (substituting } x=r^2\text{)}\\
&=\sum_j \binom{\ell-k}{j}u^{\ell-k-j}(j+k)! &\text{ (substituting } j+k=m\text{)}\\
&=\sum_m \binom{\ell-k}{m-k}u^{\ell-m}\,m!
\end{align*}
Thus, on $\V_\ell$, we have (going back to $j$ as summation index)
$$\iint e^{-z\bar z}e^{zT}u^{\L} e^{\bar z T^*}\,{\textstyle\frac{i\,dz\,d\bar z}{2\pi}}=\sum_k\sum_j \binom{\ell-k}{j-k}u^{\ell-j}\,j!\,\JS{k}{n\ell}$$

Next, we expand $g(z,u,\bar z)$ in series and integrate, recalling that the sums are finite, integrating over $\theta$ and substituting $x=r^2$ as above,
\begin{align*}
\iint e^{-z\bar z}\sum_{j,k} \frac{z^jT^j}{j!}u^{\L} \frac{{\bar z}^k{T^*}^k}{k!}\,{\textstyle\frac{i\,dz\,d\bar z}{2\pi}}
&=\sum_{j,k}\int_0^\infty\!\!\int_0^{2\pi}e^{-r^2}
r^{j+k} e^{i\theta(j-k)} \frac{T^j}{j!}u^{\L} \frac{{T^*}^k}{k!}\,{\textstyle\frac{2r\,dr\,d\theta}{2\pi}}\\
&=\sum_{j}\int_0^\infty e^{-r^2} r^{2j}\frac{T^j}{j!}u^{\L} \frac{{T^*}^j}{j!}\,2r\,dr\\
&=\sum_{j}\int_0^\infty e^{-x} x^j\frac{T^j}{j!}u^{\L} \frac{{T^*}^j}{j!}\,dx\\
&=\sum_{j}\frac{T^j{T^*}^j}{j!}u^{\L-j} \qquad\text{( moving } \L \text{ past } {T^*}^j\text{ )}
\end{align*}  \bigskip
\begin{notation}  Now it is convenient to introduce the notation
$$T_j=\frac{T^j{T^*}^j}{j!\,j!}$$
  \end{notation}

On $\V_\ell$, comparing coefficients of $u^{\ell-j}$, we have the identity
$$T_j\Bigr|_{\V_\ell}=\sum_k \binom{\ell-k}{j-k}\, \JS{k}{n\ell}\ .$$
So by binomial inversion, we have
\begin{equation}\label{eq:JS}
\JS{k}{n\ell}=\sum_j (-1)^{k-j}\binom{\ell-j}{k-j}\,T_j\Bigr|_{\V_\ell}\ .
\end{equation}
The above formulas extend the basic result of Proposition \ref{prop:js}. \bigskip

\begin{remark} 
          See \cite{BCH} for the Bargmann transform on compact groups.
\end{remark}

\begin{subsubsection}{Johnson spectrum}
From equation \eqref{eq:JS} the spectrum of the Johnson matrices follows readily from our
diagonalization of the $T_j$ operators in the $Z$-basis. Given $n$ and $N$, with $N=n-2\alpha$, we have
to rewrite the action of ${T^*}^m/m!$, say, 
\begin{align*}
({T^*}^m/m!)\phi_j&=(j(j-1)\cdots(j-m+1)/m!)\cdot\,(N+1-j)_m\,\phi_{j-m}\\
&=\binom{j}{m}\binom{N-j+m}{m}\,m!\,\phi_{j-m}
\end{align*}
in terms of $\alpha$. On the state $\phi_{\ell-\alpha}$, replacing $N$ by $n-2\alpha$, $j$ by $\ell-\alpha$, and $m$ by $j$ in the above formula, we have
$$T_j\phi_{\ell-\alpha}=\binom{\ell-\alpha}{j}\binom{n-\alpha-\ell+j}{j}\,\phi_{\ell-\alpha}$$
Thus the spectrum of $\JS{k}{n\ell}$ is given by
$$\Lambda_k^{n\ell}(\alpha)=\sum_j \binom{\ell-\alpha}{j}\binom{n-\ell-\alpha+j}{j}\binom{\ell-j}{k-j}\,(-1)^{k-j}$$
the $\alpha^{\rm th}$ eigenvalue of $\JS{k}{n\ell}$, $0\le\alpha\le n/2$.
Cf. \cite[p.\,220]{BI}. See also \cite[\S7.4]{GSL} which has a closely related approach to the Johnson spectrum using local inclusion operators.
\end{subsubsection}

\section{Conclusion}
Starting with the regular representation of the zeon algebra over a vector space of dimension $n$, we find inclusion operators that
generate an sl(2) Lie algebra. The global representation is composed of irreducible constituents transversal to the layering of the Boolean lattice
ranked by cardinality. This representation exponentiates to a group representation that has deep connections with the structure of the Boolean algebra.
We express the matrices for the Hamming and Johnson schemes in terms of the sl(2) generators originally constructed via the zeon algebra. \bigskip

One may conclude that zeon algebras are as natural as the closely related Grassmann and Clifford algebras.
This work serves to establish their importance. Zeons are already appearing in current work, e.g. \cite{SSTZ}. 
Further explorations, extensions, and applications are for the future.

\vfill\eject
\section{Appendix}
\def\changemargin#1#2{\list{}{\rightmargin#2\leftmargin#1}\item[]}
\let\endchangemargin=\endlist 

\includegraphics[scale=0.8]{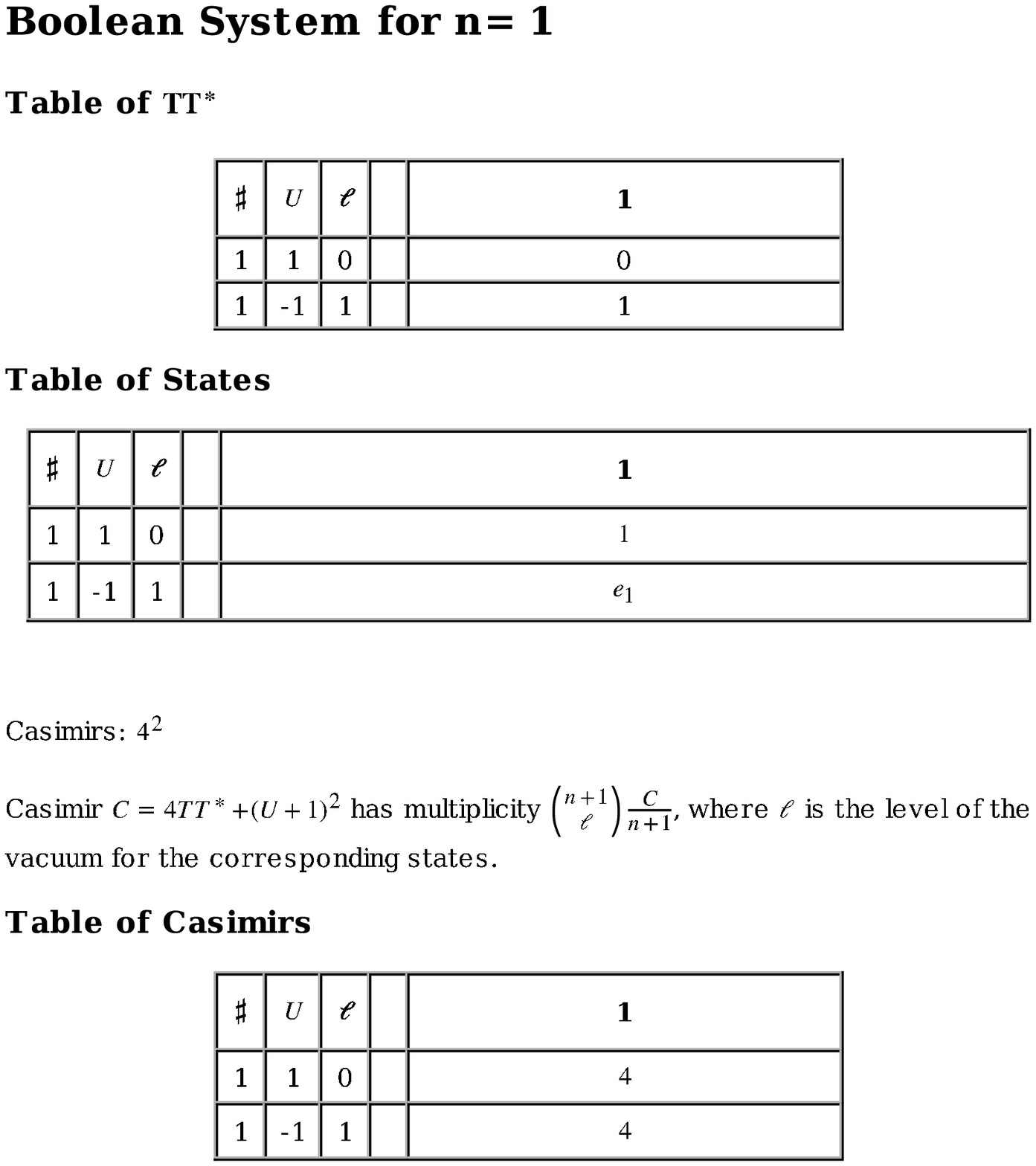}\vfill\eject
\includegraphics[scale=0.8]{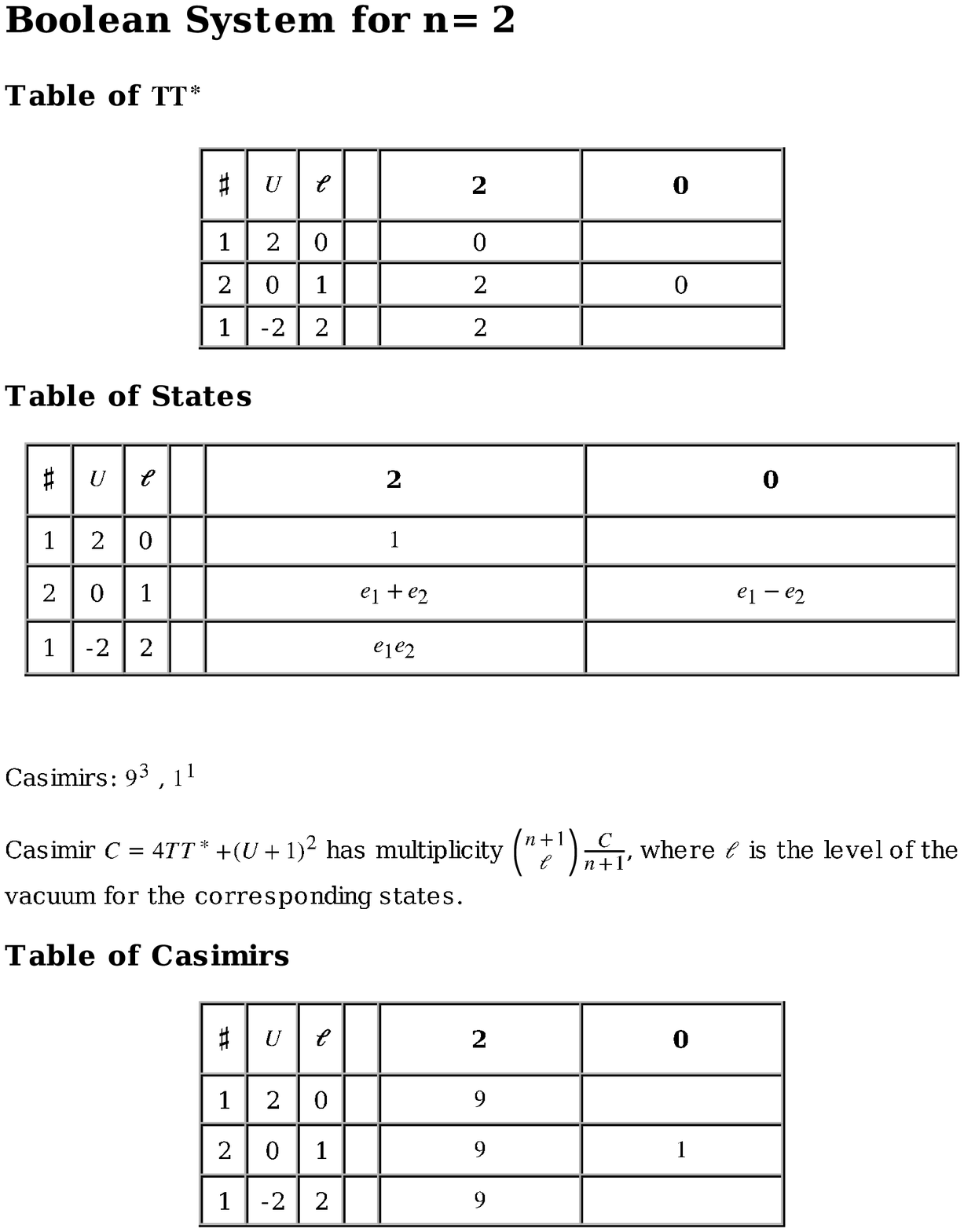}\vfill\eject
\includegraphics[scale=0.8]{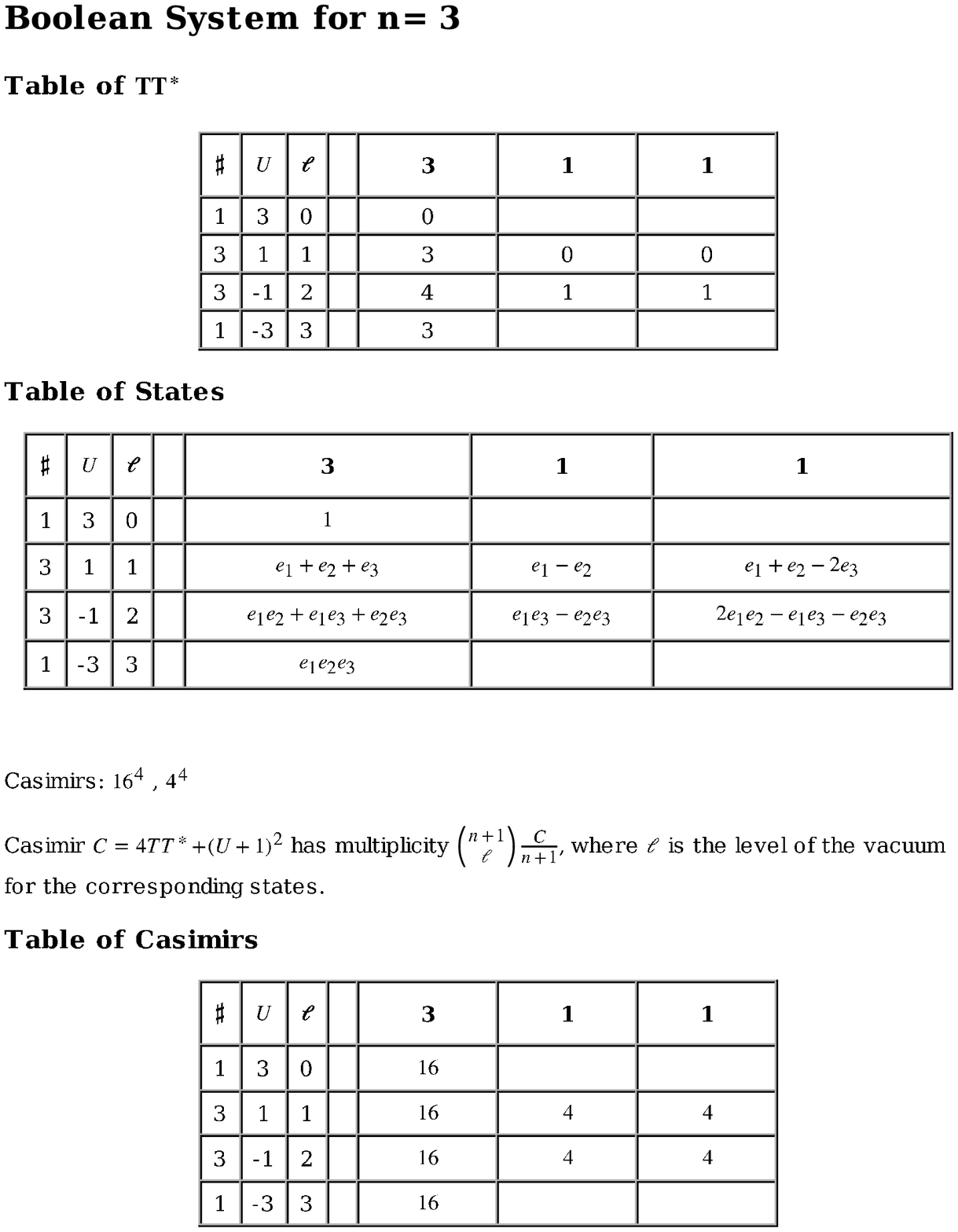}
{\voffset=-1in
\begin{changemargin}{-1.5in}{0in} 
\includegraphics[scale=.9]{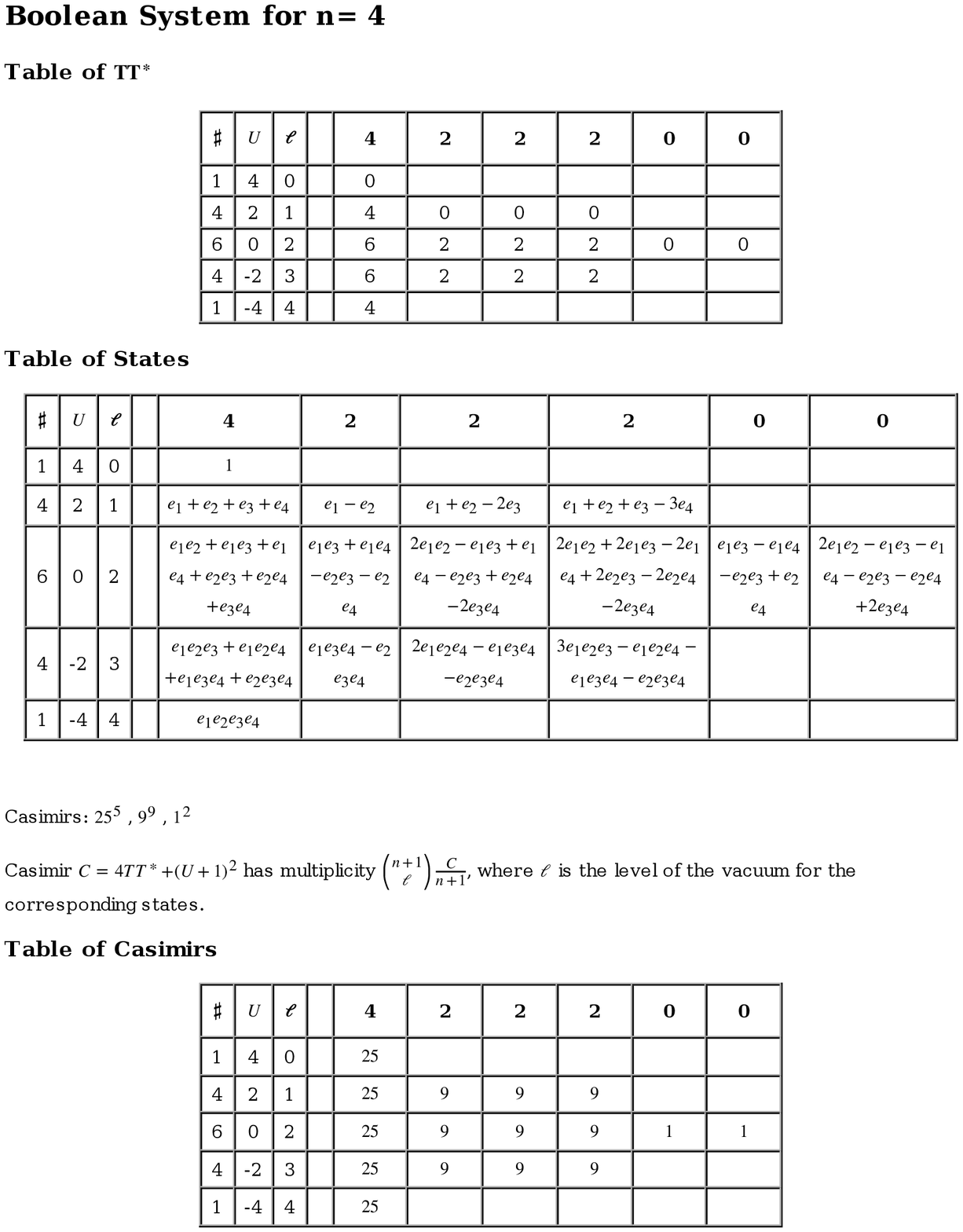}\vfill\eject
\includegraphics{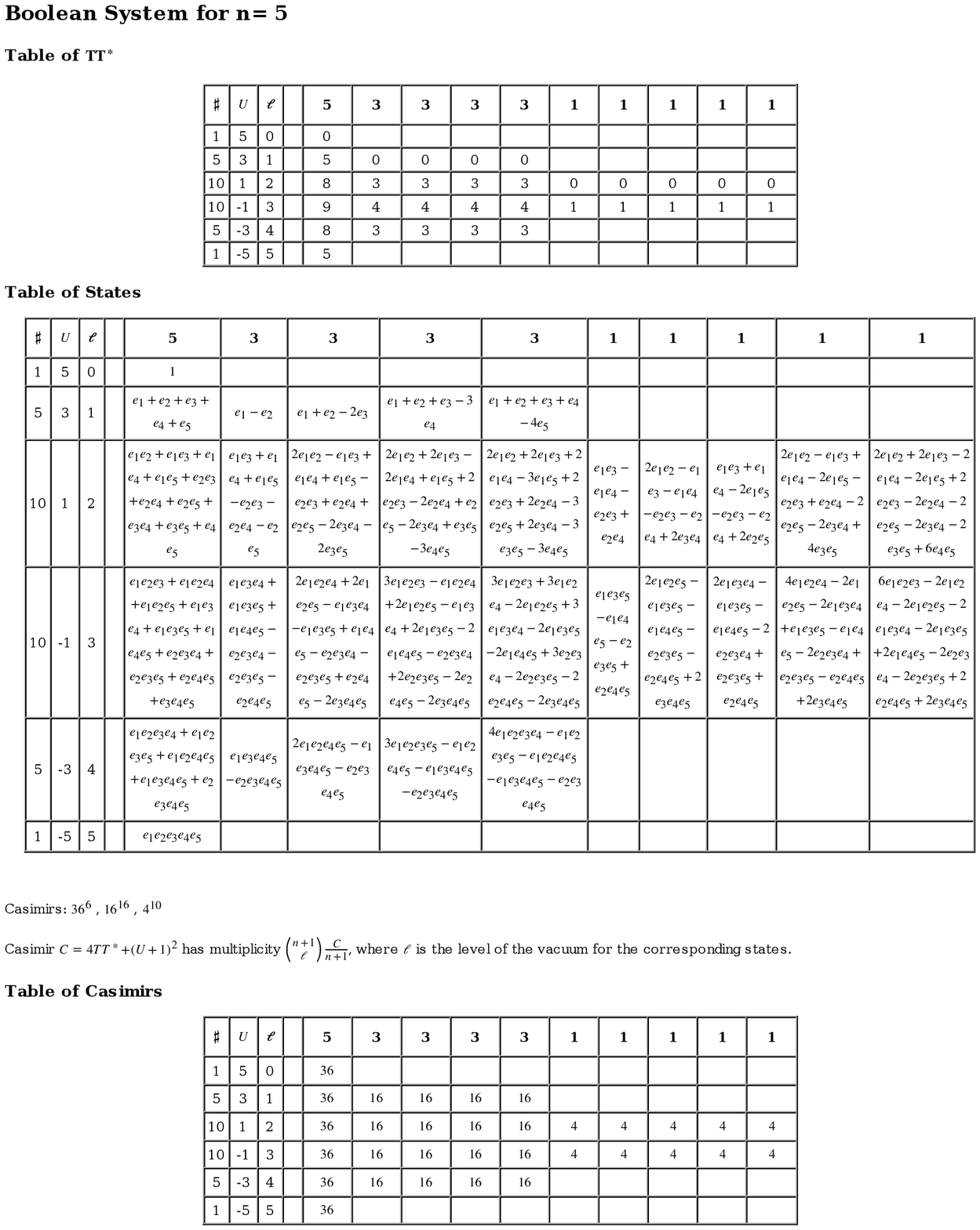}
\vfill\eject
\includegraphics[scale=.85]{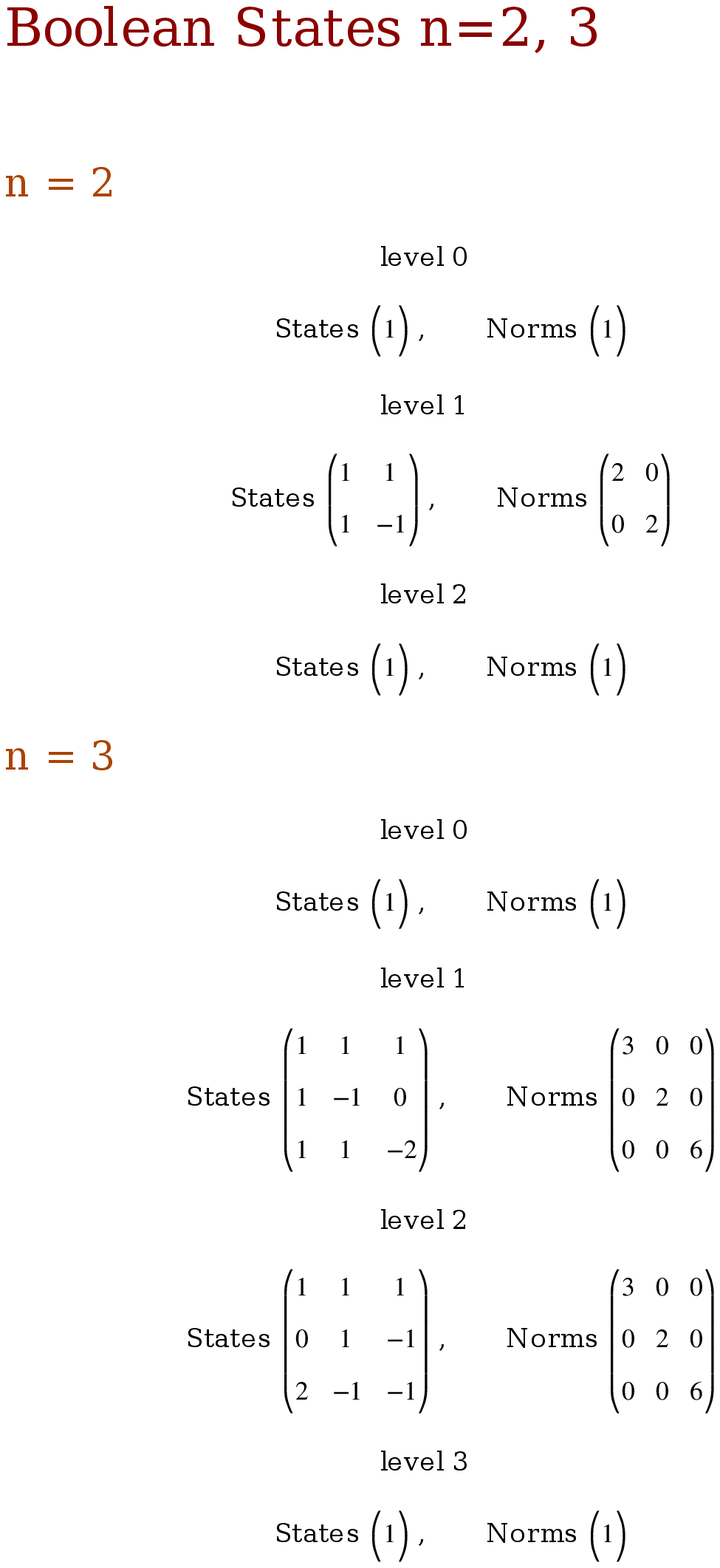}\vfill\eject
\includegraphics[scale=.9]{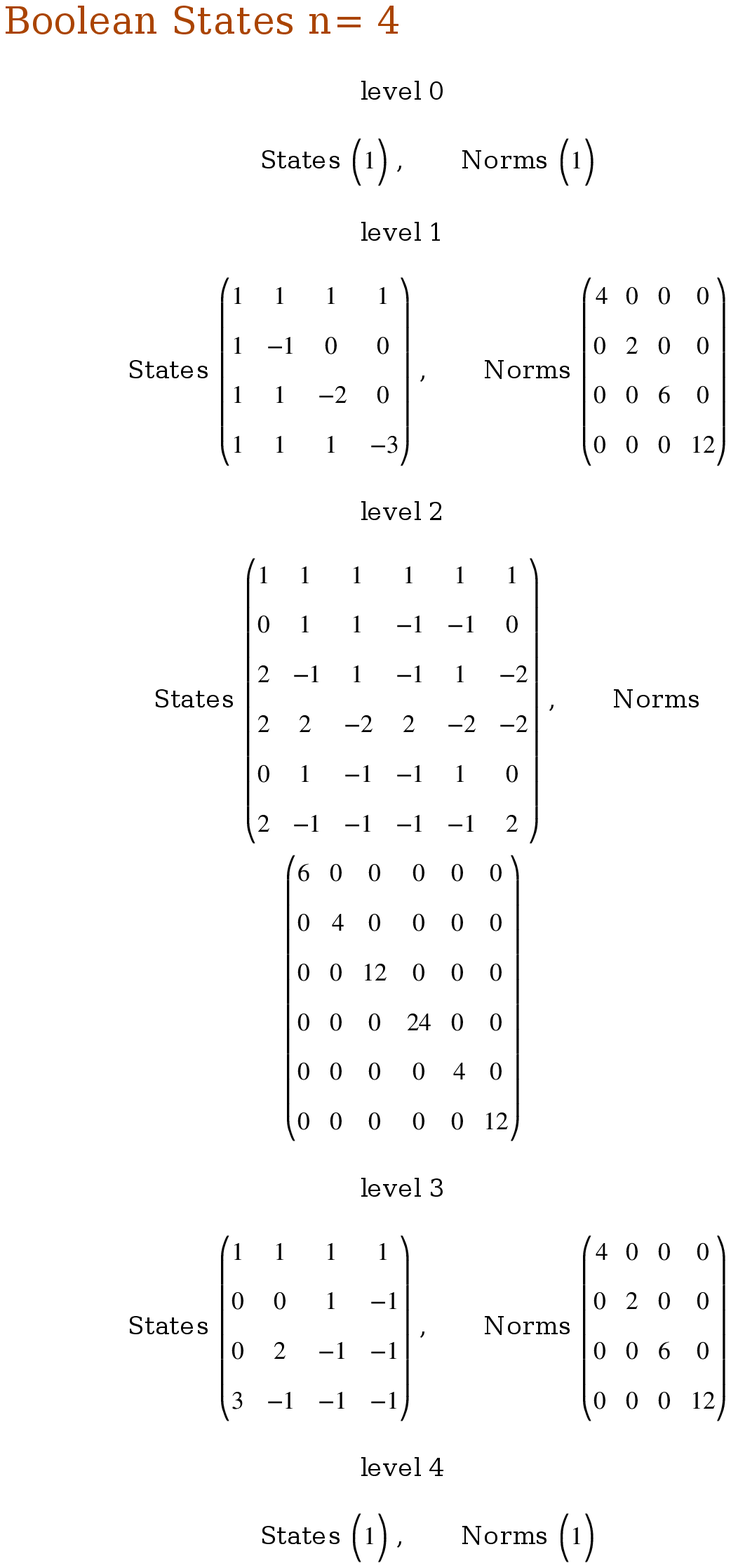}\vfill\eject
\end{changemargin}
\includegraphics[scale=.9]{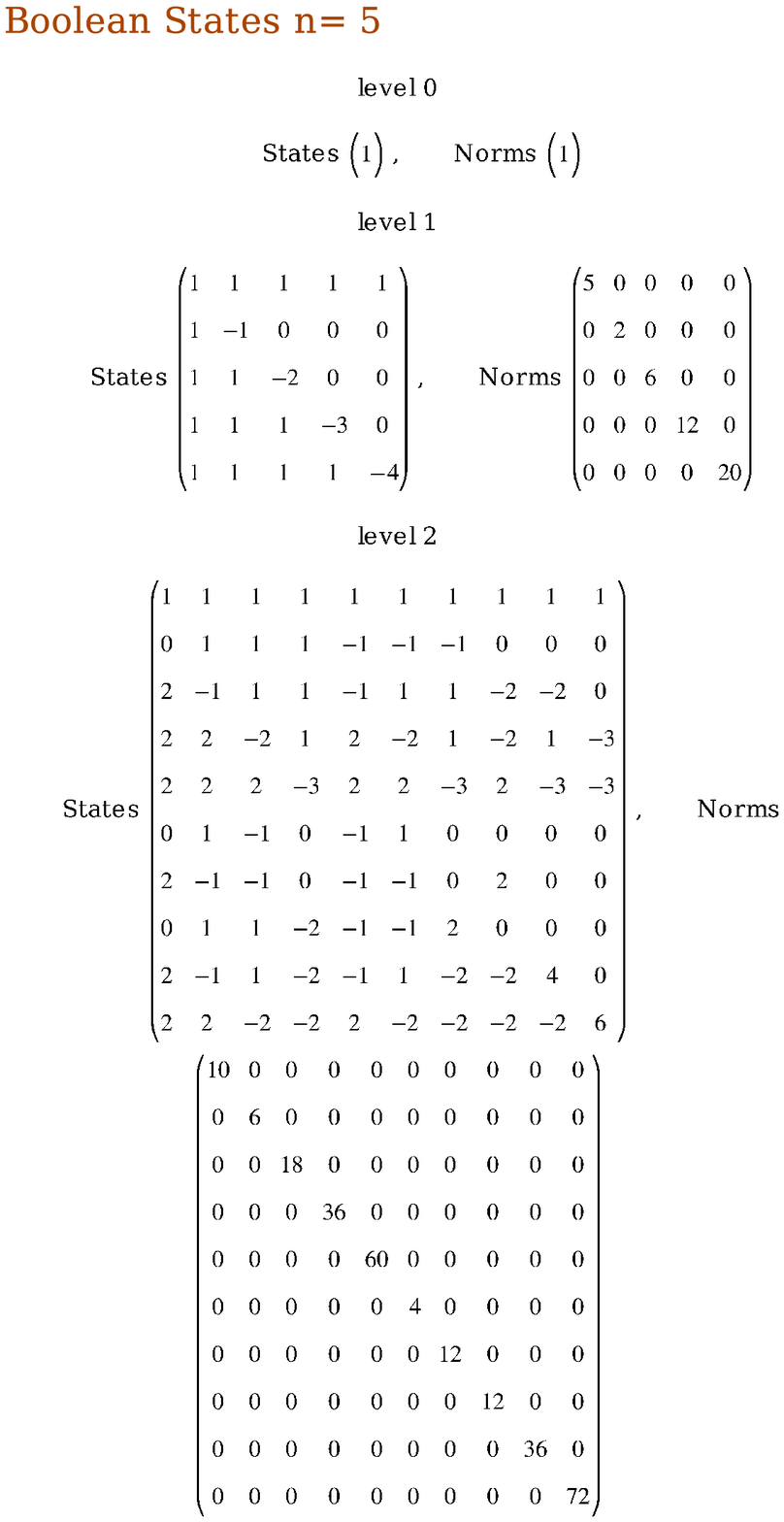}\newpage
\includegraphics[scale=.9]{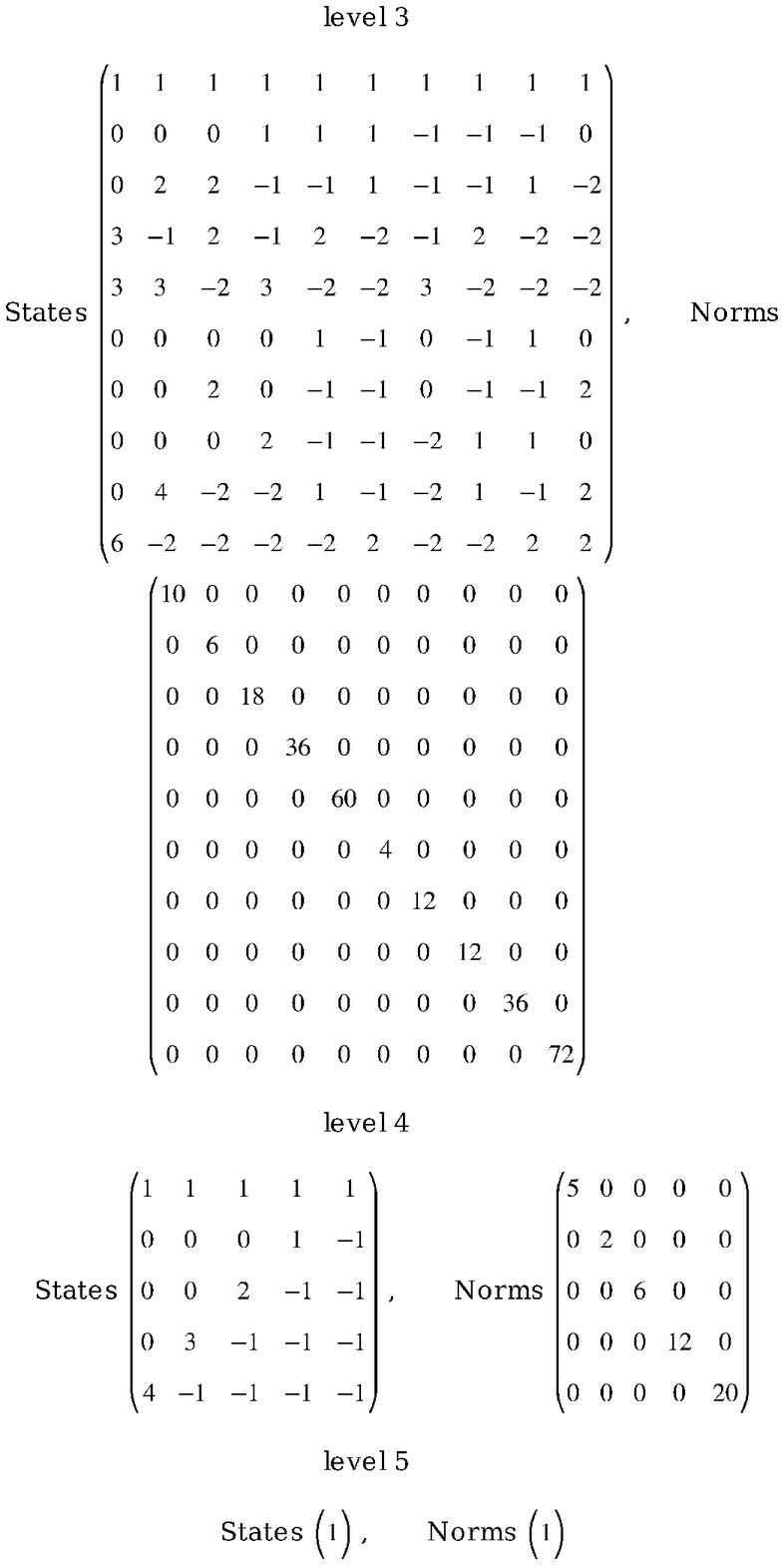}\newpage
\includegraphics[scale=.9]{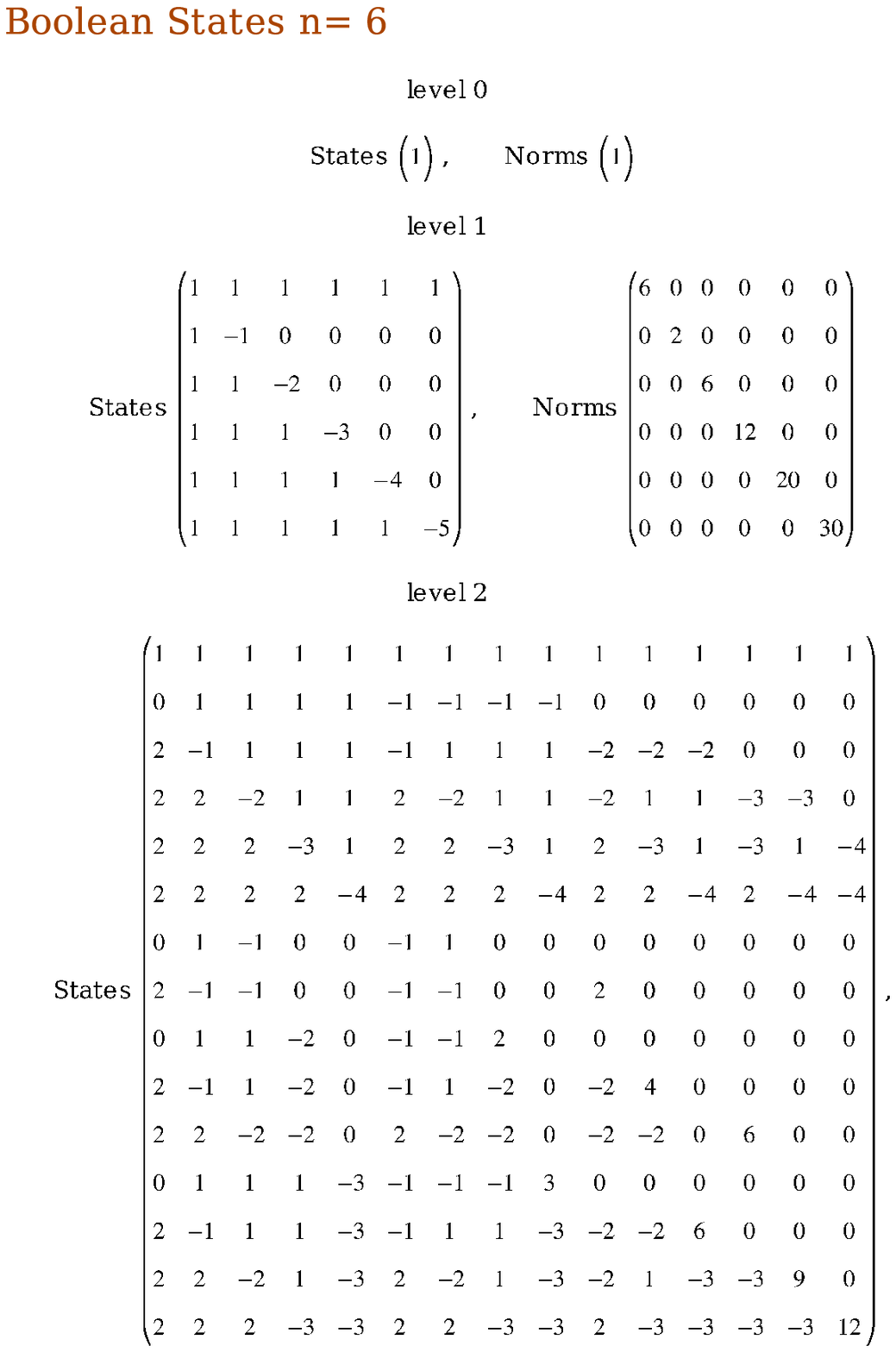}\newpage
\includegraphics[scale=.9]{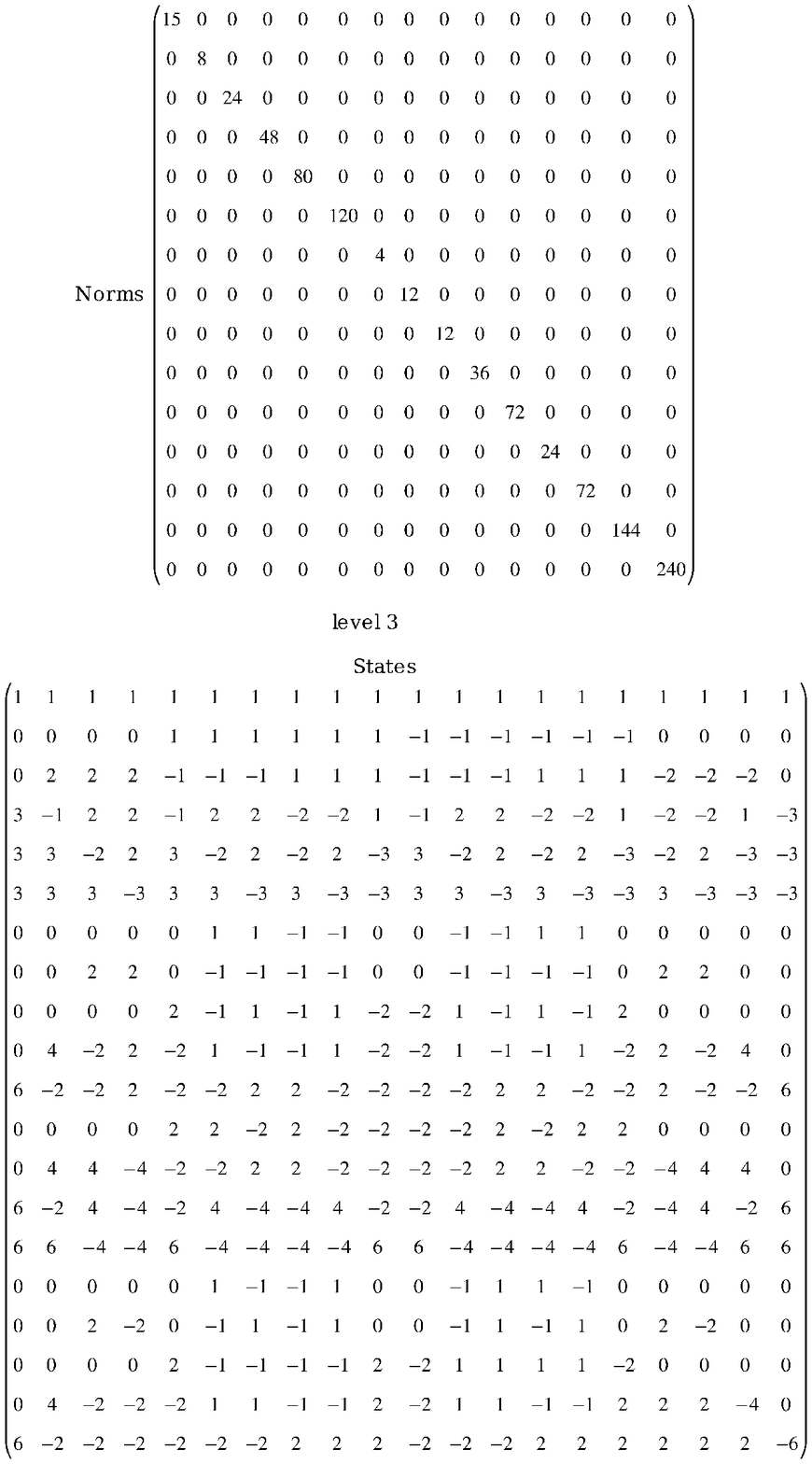}\newpage
\includegraphics[scale=.9]{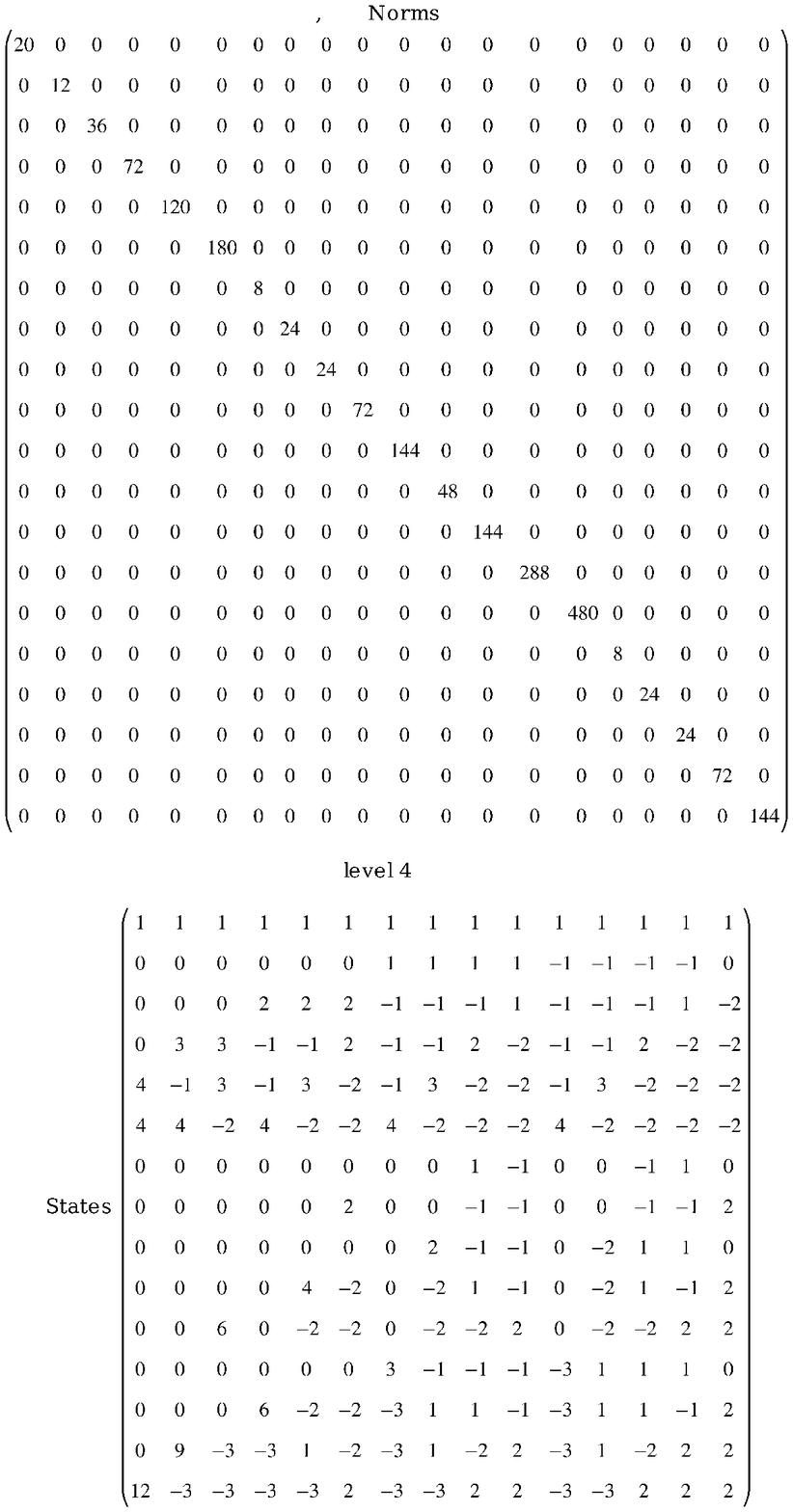}\newpage
\includegraphics{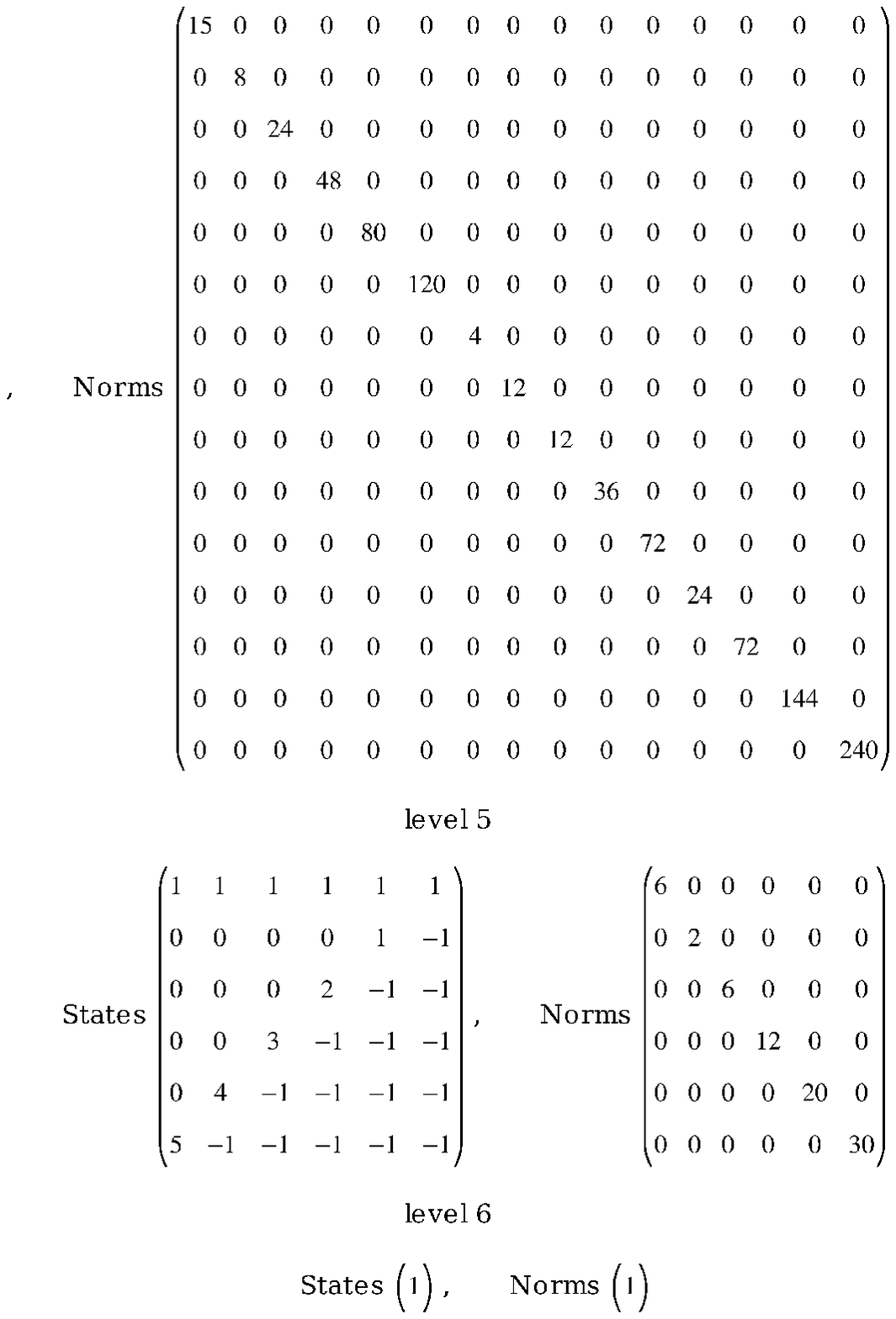}\newpage
}
\bibliographystyle{plain}
\bibliography{booleannotes}
\end{document}